\documentclass[a4paper,reqno]{amsart}

\usepackage{amsmath}
\usepackage{amssymb}
\usepackage{amsfonts}
\usepackage{amsthm}
\usepackage{mathrsfs}
\usepackage{mathtools}
\usepackage{hyperref}
\usepackage{enumerate}
\usepackage{tikz}

\hypersetup{psdextra,unicode,colorlinks=true}
\mathtoolsset{showonlyrefs=true}
\numberwithin{equation}{section}
\allowdisplaybreaks

\newtheorem{thm}{Theorem}[section]
\newtheorem{lem}[thm]{Lemma}
\newtheorem{prop}[thm]{Proposition}
\newtheorem{coro}[thm]{Corollary}
\newtheorem{defi}[thm]{Definition}
\newtheorem{remark}[thm]{Remark}

\newcommand{\me}{\mathrm{e}}
\newcommand{\mi}{\mathrm{i}}
\newcommand{\dif}{\mathop{}\!\mathrm{d}}

\begin{document}
\title{Symmetric hyperbolic Schr\"{o}dinger equations on tori}

\author{Baoping Liu}
\address{Department of Mathematics, School of Mathematical Sciences\\
	Peking University\\
	Beijing\\ China}
\email{baoping@math.pku.edu.cn}

\author{Xu Zheng}
\address{Department of Mathematics, School of Mathematical Sciences\\
	Peking University\\
	Beijing\\ China}
\email{xuzheng-math@stu.pku.edu.cn}

\thanks{The authors are supported by the NSF of China (No. 12571254, 12341102).}
\begin{abstract}

	In this paper, we study the symmetric hyperbolic Schr\"{o}dinger equations in the periodic setting.

	First, we prove  full range Strichartz estimates on  general tori by adapting Bourgain's major arc method~\cite{Bourgain1993}. The result is sharp for rational tori.

	Second, on two-dimensional rational tori, we establish optimal local well-posedness for two hyperbolic nonlinear Schrödinger (HNLS) equations: the septic HNLS and the hyperbolic-elliptic Davey–Stewartson system.

\end{abstract}
\maketitle

\section{Introduction}

The Strichartz estimate has been a key tool in the study of nonlinear Schr\"{o}dinger equations for a long time.
The Strichartz estimate for elliptic  Schr\"{o}dinger equations on Euclidean domain was established by \cite{Strichartz1977, Velo} and extended to endpoint cases by Keel and Tao~\cite{Keel-Tao1998}. The standard arguments are based on $TT^*$ method and dispersive estimates. On non-Euclidean domains one usually faces the problem of worse dispersive estimates.

Bourgain~\cite{Bourgain1993} initiated the study of Strichartz estimates on tori $\mathbb{T}^d = (\mathbb{R}/2\pi\mathbb{Z})^d$ via techniques from analytic number theory and proved that
\begin{equation}\label{Bourgain 1993 Strichartz estimate}
	\| P_{\leq N} \me^{\mi t\Delta} \phi \|_{{L_{t,x}^{p}([0,1]\times \mathbb T^{d})}} \lesssim N^{\frac{d}{2} - \frac{d+2}{p}} \| \phi \|_{L^2_x(\mathbb T^d)}, \quad \forall p\geq \frac{2(d+4)}d,
\end{equation}
where $P_{\leq N}$ denotes the Littlewood-Paley projection operator. Later, by Fourier decoupling method, Bourgain-Demeter~\cite{bour2015} proved the full range Strichartz estimates on general $d$-dimensional tori with extra $N^\varepsilon$ loss
\begin{equation}\label{Bourgain-Demeter Strichartz estimate by decoupling}
	\| P_{\leq N} \me^{\mi t\Delta} \phi \|_{{L_{t,x}^{p}([0,1]\times \mathbb T^{d})}} \lesssim_\varepsilon N^{\frac{d}{2} - \frac{d+2}{p}+\varepsilon } \| \phi \|_{L^2_x(\mathbb T^d)}, \quad \forall p\geq \frac{2(d+2)}d,\ \forall \varepsilon>0.
\end{equation}
The  loss in \eqref{Bourgain-Demeter Strichartz estimate by decoupling} was  removed by Killip-Vi\c{s}an~\cite{Killip2016} for $p>2(d+2)/d$. Recently, Herr-Kwak~\cite{Herr-Kwak} proved the sharp  endpoint $L^4$ estimate on $\mathbb T^2$  via incidence geometry
\begin{equation}
	\| P_{\leq N} \me^{\mi t\Delta} \phi \|_{L^4_{t,x}([0,1]\times \mathbb T^2)} \lesssim (\log N)^{1/4} \| \phi \|_{L^2_x(\mathbb T^2)}.
\end{equation}

Other geometric settings have also been extensively studied in recent years.  Takaoka-Tzvetkov~\cite{Takaoka-Tzvetkov2002} established a lossless $L^4$ Strichartz estimate on $\mathbb{R}\times\mathbb{T}$.  Burq-G\'{e}rard-Tzvetkov~\cite{BGT2004} proved Strichartz estimates with fractional loss of derivatives on any compact Riemannian manifold. Huang-Sogge~\cite{Sogge2} obtained optimal estimate  on Zoll manifold.   See also~\cite{Sogge1, Sogge3} and references therein for results in other settings.

Now we turn our attention to the hyperbolic Schr\"{o}dinger equation
\begin{equation}\label{LHLS}
	(\mi \partial_t + \Box ) u =0,  \end{equation}
with $\Box = \theta_1\partial^2_{x_1}+\dots+\theta_v\partial^2_{x_v}-\theta_{v+1}\partial^2_{x_{v+1}}-\dots-\theta_d\partial^2_{x_d}$, $ \theta_j\in (0,1], 1\leq v\leq d/2.$~\footnote{As in~\cite{Killip2016}, the general flat tori are considered as $\mathbb{R}^d/(L_1\mathbb{Z}_1\times L_2\mathbb{Z}_2\times \dots \times L_d\mathbb{Z}_d)$ with $L_j\in \mathbb{R}^{+}$. But we can incorporate the geometry into the  operator by taking $\theta_j=L_j^{-2}$ in  the definition of $\Box$ and considering the problem on the standard torus. We can further restrict to $\theta_j\in (0,1]$ by a change of variable in time.}\label{footnote1}

In \cite{bourgain2017}, Bourgain-Demeter proved the estimate on $\mathbb{T}^d$\begin{equation}\label{Bourgain-Demeter hyperbolic Strichartz estimate by decoupling}
	\| P_{\leq N} \me^{\mi t\Box} \phi \|_{L^p_{t,x}([0,1]\times \mathbb T^d)} \lesssim_\varepsilon N^{\beta_{d,v}(p)+\varepsilon} \| \phi \|_{L^2_x(\mathbb T^d)}, \quad \forall p\geq 2,\ \forall \varepsilon>0,
\end{equation}
here
\begin{equation}
	\beta_{d,v}(p) = \max\left\{ \frac{d}{2} - \frac{d+2}{p}, \frac v2-\frac vp\right\}.\label{definition of powerindex}
\end{equation}
Notice that estimate \eqref{Bourgain-Demeter hyperbolic Strichartz estimate by decoupling} is weaker than the elliptic case \eqref{Bourgain-Demeter Strichartz estimate by decoupling} when $p<\frac{2(d-v+2)}{d-v}$. In particular, for the square torus, this is due to the fact that the hyperbolic paraboloid contains a vector
subspace of dimension $v$.

The extra $N^\varepsilon$ loss in~\eqref{Bourgain-Demeter hyperbolic Strichartz estimate by decoupling} was removed by \cite{Başakoğlu-Wang2025-3} in the partial range $p>\frac{2(d-v+2)}{d-v}$, following the argument of Killip-Vi\c{s}an~\cite{Killip2016}; and by \cite{LZ2025} in the full range $p\geq 2$ for $d=3,\ \theta_j=1$ via incidence geometry approach, motivated by~\cite{Herr-Kwak}.  See also earlier result of Godet-Tzvetkov \cite{Tzvetkov} and Wang \cite{wang2013},  where the $L^4$ Strichartz estimate with $1/4$-derivative loss on square tori $\mathbb{T}^2$ was obtained
using different methods.

Here we consider the case $d$ is even and $v=\frac{d}{2}$, i.e. the symmetric hyperbolic Schr\"{o}dinger equation,  and   remove the $N^\varepsilon$-loss in~\eqref{Bourgain-Demeter hyperbolic Strichartz estimate by decoupling} in the full range on general tori.  Our first main result reads as follows:
\begin{thm}\label{thm:main sharp Strichartz estimate on d-dimensional tori}
	For $v=\frac{d}{2}$ and any $\phi\in {L^2_{x}(\mathbb T^{d})}$, we have that
	\begin{equation}\label{est:main sharp Strichartz estimate on d-dimensional tori}
		\| P_{\leq N} \me^{\mi t\Box} \phi \|_{{L_{t,x}^{p}([0,1]\times \mathbb T^{d})}} \lesssim \big( N^{\frac d2-\frac{d+2}p} + N^{\frac d4-\frac d{2p}} \big) \| \phi \|_{{L^2_{x}(\mathbb T^{d})}},\quad \forall p\geq 2.
	\end{equation}
\end{thm}
\begin{remark}
	Note that when  $d$ is even, and $v=\frac{d}{2}$,
	the turning point $B$ in the following picture corresponds to
	\[p=\frac{2(d-v+2)}{d-v}= \frac{2(d+4)}{d},\]
	which coincides exactly with the exponent in \eqref{Bourgain 1993 Strichartz estimate}.  This observation motivates us to employ the major arc approach from the seminal paper~\cite{Bourgain1993}.

	We will thus focus our proof for $p=\frac{2(d+4)}{d}$, and the general case follows from
	interpolation with $L^2$ and $L^\infty$ estimates.
\end{remark}
\begin{figure}[htbp]
	\begin{tikzpicture}[xscale=3,yscale=2]
		\draw[->] (0,0) -- (1,0);
		\draw[->] (0,0) -- (0,3/2);
		\draw (0,1) -- (1/4,1/4);
		\draw (1/4,1/4) -- (1/2,0);
		\draw[dashed] (0,1/2) -- (1/4,1/4);
		\draw[dashed] (1/3,0) -- (1/4,1/4);
		\node[below] at (1,0) {$\frac1p$};
		\node[left,font=\small] at (0,3/2) {$\beta_{d,v}(p)$};
		\node[below] at (1/2,0) {$\frac 12$};
		\node[above, font=\tiny] at (1/2,0) {C};
		\node[above, font=\tiny] at (1/3,1/6) {B};
		\node[below] at (1/4,0) {$\frac{d}{2(d+2)}$};
		\node[left] at (0,1/2) {$\frac v2$};
		\node[left] at (0,1) {$\frac d2$};
		\node[right,font=\tiny] at (0,1) {A};
		\foreach \x in {1/3,1/2}{\draw (\x,0) -- (\x,0.03); }
		\foreach \y in {1/2,1}{\draw (0,\y) -- (0.02,\y); }
	\end{tikzpicture}
	\caption{Relation between $\beta_{d,v}(p)$ and $p$}
	\label{figure: relation between power index and p}
\end{figure}
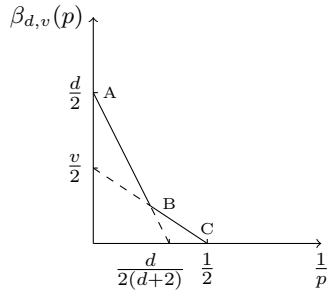

Next,  we consider the Cauchy problem for nonlinear hyperbolic Sch\"{o}dinger equation (HNLS)
\begin{equation}
	\mi\partial_t u + \Box u = \pm |u|^{2k}u,\quad (t,x)\in \mathbb{R}\times \mathbb{T}^d,
	\label{HNLS on Td}
\end{equation}

HNLS  serves as a fundamental model in multiple physics contexts, such as plasma waves~\cite{Plasma-1,Plasma-2,Plasma-3, Plasma-4} and gravity water waves~\cite{Craig, D-S, Saut,Totz2015cmp}.
We note that the symmetric hyperbolic Schr\"{o}dinger equation appears not only in two dimensions~\cite{Totz2015cmp} but also in higher dimensions. Examples include its emergence from the Boltzmann equation via the Wigner transform~\cite{Natasa1, Natasa2}, or applying  Fourier transform in the velocity variable — albeit in this case yielding $(\mi\partial_t +\nabla_{\xi}\cdot\nabla_x)\tilde{f}=0$~\cite{Chen_2024,Chenxuwen2,Chenxuwen3,Wangyuzhao2024}.

The mathematical problems on HNLS equations were first addressed by Ghidaglia and Saut in \cite{Ghidaglia-Saut1993,Ghidaglia-Saut1996}.
We refer the readers to the survey paper by Saut-Wang~\cite{Saut2024hnls} for more details.

In the study of \eqref{HNLS on Td} on Euclidean domain, the critical Sobolev index $s_c=\frac d2-\frac 1k$ is of particular interest, since the scaling symmetry $u \mapsto \lambda^{1/k} u(\lambda^2t,\lambda x)$ leaves the Sobolev norm $\| \cdot \|_{\dot H^{s_c}(\mathbb{R}^d)}$ invariant.  Since the dispersive estimate and Strichartz estimates are the same as the elliptic Schr\"{o}dinger equation, sharp local well-posedness with data in $H^{s_c}(\mathbb{R}^d)$  follows directly from the Strichartz estimate.
The Cauchy problem for HNLS on  tori has been studied in~\cite{wang2013, LZ2025, Başakoğlu-Wang2025-3}. We quickly summarize the results (when equation is local well-posedness for data in $H^s(\mathbb{T}^d)$, we shorten the notation by saying LWP in $s$. Also notice for hyperbolic equation we always have $1\leq v\leq \frac{d}{2}$.)

\begin{itemize}
	\item Subcritical results:  LWP for $s>s_c$ when $(d,v, k)\not=(2,1, 1)$;  LWP for $s>\frac12$ when $(d,v, k)=(2,1, 1)$.
	\item Critical results: LWP for $s\geq s_c$ when    $(d,v, k) \neq (2,1,1)$, $(2,1, 2)$, $(2,1,3)$, $(3,1,1)$, $(4,2,1)$.
\end{itemize}
The above results hold for general tori~\cite{Başakoğlu-Wang2025-3}, though in some cases the Cauchy problem  was considered first on square torus~\cite{wang2013, LZ2025}.

On rational tori,  we also have ill-posedness in the sense that the first Picard iteration is unbounded.
\begin{itemize}
	\item $(d,v,k)=(2,1,1)$:  ill-posedness for $s=\frac12$~\cite{LZ2025}, see also~\cite{wang2013} for $s<\frac12$.
	\item $(d,v,k)=(3,1,1)$:  ill-posedness for $s=s_c$~\cite{LZ2025}.
\end{itemize}

Here  we deal with all the leftover cases at critical regularity on rational tori by proving that
\begin{thm} \label{thm:cauchyProblemHNLS} For initial data $u_0\in H^{\frac d2-\frac1k}(\mathbb{T}^d)$, the Cauchy problem for HNLS $\eqref{HNLS on Td} $ $(\mbox{with all } \theta_j=1)$ is locally well-posedness for $(d,v, k)=(2,1, 3)$ and ill-posedness for $(d,v, k)=(2,1, 2)$ or $(4,2, 1)$ in the sense that the solution flow is not $C^{2k+1}$ at the origin.
\end{thm}
\begin{remark}
	On $\mathbb{R}\times\mathbb{T}$, the sharp $L^4$ estimates without loss were proved in~\cite{Deng-Fan-Zhao2025}, which implies sharp local well-posedness for cubic HNLS with small $L^2$ data.  It extends to global well-posedness by the mass conservation law.
\end{remark}

We also consider
the Davey-Stewartson system (DS) \begin{equation}
	\begin{cases}
		\mi \partial_t u + \sigma_1 \partial_{x_1}^2 u + \partial_{x_2}^2 u = \sigma_2 |u|^2u + (\partial_{x_1} \psi)u , \\
		\alpha \partial_{x_1}^2 \psi + \partial_{x_2}^2 \psi = - \gamma \partial_{x_1}(|u|^2),
	\end{cases} \quad (t,x) \in \mathbb  R\times \mathbb T^2.
	\label{DSII}
\end{equation}
Depending on the signs $(\operatorname{sgn}(\sigma_1), \operatorname{sgn}(\alpha))=(1,1), (1,-1), (-1, 1)$ and $ (-1,-1)$, the Davey-Stewartson systems are usually classified as elliptic-elliptic, elliptic-hyperbolic, hyperbolic-elliptic, and hyperbolic-hyperbolic. The elliptic-elliptic and elliptic-hyperbolic Davey-Stewartson systems were first derived by Davey-Stewartson~\cite{D-S} in the context of water waves, without surface tension (or capillarity).
Later, Djordjevic and Redekopp~\cite{Redekopp} showed that the inclusion of surface tension leads to a negative parameter
$\alpha$ when capillary effects are dominant, which corresponds to hyperbolic-elliptic system.
A rigorous derivation of \eqref{DSII} from the water wave problem in the modulational scaling regime was provided by Craig-Schanz-Sulem~\cite{Craig}.
Independently, Ablowitz-Haberman~\cite{Haberman}, Morris~\cite{Morris} and Cornille~\cite{Cornille} obtained a particular form of \eqref{DSII} as an example of completely integrable model  that generalizes  the 1d nonlinear Schr\"{o}dinger equation to two dimensions.
These are usually referred to as the DSI and DSII systems, corresponding to the parameters
\begin{align*}
	(\sigma_1,\sigma_2,\alpha,\gamma)=  & (1,\pm1,-1,\mp 2)  \\
	(\sigma_1, \sigma_2,\alpha,\gamma)= & (-1,\pm1,1,\pm 2).
\end{align*}

The investigation of the system \eqref{DSII} on $\mathbb R^2$ in terms of well-posedness was initiated by Ghidaglia-Saut~\cite{Saut} who established the local well-posedness in the elliptic–elliptic, elliptic–hyperbolic and hyperbolic–elliptic cases. Linares-Ponce~\cite{Ponce} studied the local well-posedness in weighted Sobolev spaces for the elliptic-hyperbolic and hyperbolic–hyperbolic cases.
Nachman-Regev-Tataru~\cite{Tataru2020} considered the defocusing DSII equation,  i.e. $(\sigma_1, \sigma_2,\alpha,\gamma)=(-1,1,1, 2)$, and
proved global well-posedness and scattering for any initial data in $L^2(\mathbb R^2)$.

The periodic problem is much less studied. $L^4$ Strichartz estimate would imply local well-posedness for elliptic-elliptic DS with data in $H^s(\mathbb{T}^2)$ for $s>0.$ Godet~\cite{Godet} obtained  local well-posedness  for the hyperbolic–elliptic case in $H^s(\mathbb T^2)$ for $s > 1/2$ as well as a lower bound on the  blow-up rate for this equation.

In this paper, we study the hyperbolic-elliptic Davey-Stewartson system with  $\sigma_1 = -1$ and  $\alpha>0$.
The system \eqref{DSII} can be rearranged as a hyperbolic Schr\"{o}dinger equation with nonlocal nonlinearity
\begin{equation}
	\label{DS hyperbolic nonlocal}
	\mi\partial_t u  -\partial_1^2u+\partial_2^2 u = \Big(\sigma_2-\frac\gamma{1+\alpha} \Big) |u|^2 u +  \frac{\gamma}{1+\alpha} \mathscr{J}_{\mathrm{DS}}(|u|^2)\cdot u,
\end{equation}
where $\mathscr{J}_{\mathrm{DS}}$ is the multiplier operator with symbol $m_{\mathrm{DS}}$ given by
\begin{equation}\label{def:JDS}
	m_{\mathrm{DS}}(\xi,\eta) = \frac{-\xi^2+\eta^2}{\alpha \xi^2+\eta^2}, \quad (\xi,\eta) \neq (0,0);  \quad m_{\mathrm{DS}}(0,0)=1.
\end{equation}
\begin{thm}\label{thm:LWP for DS} We have the following results about well-posendess of \eqref{DS hyperbolic nonlocal}:
	\begin{enumerate}
		\item The Cauchy problem for \eqref{DS hyperbolic nonlocal} is locally well-posedness in $H^s(\mathbb{T}^2)$ for $s>0$ if $\sigma_2=\gamma/(1+\alpha)$ and $s>1/2$ if $\sigma_2\neq \gamma/(1+\alpha)$.
		\item If $\sigma_2\neq \gamma/(1+\alpha)$, let $T > 0$ be an arbitrarily small positive constant. Assume that the data-to-solution map $ u_0 \mapsto u(\cdot)$ associated with \eqref{DS hyperbolic nonlocal} on smooth data extends continuously to a map from  $H^{1/2}(\mathbb T^2)$ into $C([0,T);H^{1/2}(\mathbb T^2))$. Then this map will not be $C^3$ at the origin.
	\end{enumerate}
\end{thm}
\begin{remark}
	The  DSII system satisfies the condition $\sigma_2=\gamma/(1+\alpha)$ and hence enjoys better local well-posedness, while for $\sigma_2\neq\gamma/(1+\alpha)$ the DS system behaves like cubic hyperbolic Schr\"{o}dinger equation~\cite{wang2013}.
\end{remark}
\begin{remark}

	Although Theorems~\ref{thm:cauchyProblemHNLS} and~\ref{thm:LWP for DS} are stated for square tori, the results generalize straightforwardly to any rational tori, i.e., those for which $L_j/L_k\in \mathbb{Q}, $
	(see footnote~\ref{footnote1}).
	The essential ingredient in our proof is a refined bilinear estimate for terms with an off-diagonal nature; see Definition~\ref{def:off-diagonal}, Lemma~\ref{lemv2: Bilinear L2 estimate for projection on diagonal}, and \ref{lem:refinedLp}.  It is a rather interesting question whether one can obtain same well-posedness result for septic NLS and the DSII equation on 2d irrational tori.
\end{remark}

The outline of this paper is as follows: In Section~\ref{Section: Strichartz Estimates}, we prove the Strichartz estimate using the major arc method, and we adopt the simplified approach of Killip-Vi\c{s}an~\cite{Killip2016}. In Section~\ref{Sec:bilinear}, we single out the off-diagonal part for the bilinear term and prove the refined bilinear estimate. In Section~\ref{Sec:LWP}, we set up the general framework using $U,V$ space.  Then we prove an improvement of the trilinear estimate based on the refined bilinear estimate.
Finally, we prove Theorem~\ref{thm:cauchyProblemHNLS} and~\ref{thm:LWP for DS} via the contraction mapping argument.

\section{Strichartz Estimates}\label{Section: Strichartz Estimates}
\subsection{Notations}
We denote $A\lesssim B$ or $A=O(B)$ if $A\leq CB$ holds for some constant $C>0$ independent with $A$ and $B$. We write $A\approx B$ if $A\lesssim B$ and $B\lesssim A$.

For $f \in L^2(\mathbb T^d)$, the Fourier coefficients of $f$ are given by
\[ \hat f(k) = \int_{\mathbb T^d} f(x) \me^{-2\pi\mi k\cdot x} \dif x, \quad k\in\mathbb{Z}^d,\]
and the Fourier series of $f$ is
\[ f(x) =  \sum_{k\in\mathbb Z^d} \hat f(k) \me^{2\pi\mi k\cdot x}.\]
For any subset $S\subset \mathbb Z^d$, we denote $P_S$ for the Fourier multiplier with symbol $\chi_S$, i.e.
\[ P_Sf = \sum_{k\in S} \hat f(k) \me^{ 2\pi\mi k\cdot x}. \]
When we discuss cubes (or squares) $C\subset \mathbb Z^d$, we always assume  their edges are parallel to the coordinate axes.

In this paper, $N$ will always be a dyadic integer, i.e.  $N=2^n$ for some $n\in\mathbb N$.
For $S = {\{ k \in \mathbb Z^d \mid  N \leq |k| < 2N \}}$ we simply write $P_S$ as $P_N$, and
\[P_{\leq N} f= \sum_{M\leq N,\ M \text{ dyadic}}P_M f.\]

For $s\in\mathbb R$, the Sobolev space $H^s(\mathbb T^d)$ is the set of all functions $f\in L^2(\mathbb T^d)$ such that the norm
\[ \| f \|_{H^s(\mathbb{T}^d) }: = \left( \sum_{k\in\mathbb{Z}^d} \left<k\right>^{2s} |\hat f(k)|^2 \right)^{1/2} \]
is finite, where $\left< k \right> = \sqrt{1+|k|^2}$.

For any Fourier multiplication operator $\mathscr{A}$, we naturally view it as a  sesquilinear operator by denoting
$\mathscr{A}(u_1, u_2)= \mathscr{A}(u_1\bar{u}_2)$.

For integers $a,b$, we denote $a|b$ if $a^{-1}b\in\mathbb Z$.
By $a\sim b$ we  mean $a^{-1}b\in [1/2 , 2]$.
$d(a)$ is the number of divisors of $a$.
$d(a;Q)$ is the number of divisors $q$ of $a$ that obey $q\sim Q$.
$\gcd(a,b) $ denotes the greatest common divisor of $a$ and $b$.
$\operatorname{lcm}(a,b)$ denotes the least common multiple of $a$ and $b$.
We also denote $\#(\cdot)$ for the counting measure and $|\cdot|$ for the Lebesgue measure,
and $\langle\cdot,\cdot\rangle$ for the $L^2$ inner product.

\subsection{Strichartz Estimate}
Consider the symmetric hyperbolic Schr\"{o}dinger equation with the operator
\begin{equation}
	\Box = \sum_{j=1}^d (-1)^{j-1}\theta_{j}\partial_{x_j}^2,\quad \theta_j\in(0,1],\quad d \text{ is even integer}.\end{equation}
We denote the one-dimensional convolution kernel of Schr\"{o}dinger propagator with smooth cut-off
\begin{equation}
	K^\ominus_N(t,x) = \sum_{n\in\mathbb{Z}} \Psi(n/N) \me^{2\pi\mi (nx+n^2t)},\quad x\in\mathbb{T}^1,
\end{equation}
where $\Psi$ is a suitable compactly supported even function. Then the $d$-dimensional convolution kernel with respect to $\{\me^{\mi t\Box}\}_{t\in\mathbb{R}}$ is given by
\begin{align}
	K_N(t,x) & =
	\prod_{j\text{ even}} \sum_{n_j}  \Psi(\frac{n_j}{N})\me^{2\pi \mi (x_jn_j + \theta_jn_j^2 t)} \prod_{j\text{ odd}} \sum_{n_j}  \Psi(\frac{n_j}{N})\me^{2\pi \mi (x_jn_j - \theta_jn_j^2 t)} \\
	         & = \prod_{j\text{ even}} K^\ominus_N(\theta_{j}t,x_j) \prod_{j\text{ odd}} \overline{K^\ominus_N(\theta_{j}t,x_j)},\quad x\in \mathbb{T}^d.
\end{align}
Here, we replaced $n_j$ by $-n_j$ when $j$ is odd, exploiting that $\Psi$ is even.

Notice that the complex conjugation doesn't affect the size of the kernel,  and the arithmetic–geometric
mean trick used by  Killip-Vi\c{s}an~\cite{Killip2016} helps to handle the general tori. Thus, we can expect the argument in~\cite{Bourgain1993} should still work for the hyperbolic Schr\"{o}dinger equation on general tori.

The solution to the free hyperbolic Schr\"{o}dinger equation with initial data $u_0$ is given by
\begin{equation}
	P_{\leq N} u(t,x) =P_{\leq N} \me^{\mi t\Box} u_0(x) = u_0*_{x}K_N(t,x),
\end{equation}
which admits the mass conservation law
\begin{equation}
	\|  u(t,\cdot)  \|_{L^2_{x} (\mathbb T^d)} = \| u_0 \|_{L^2_{x} (\mathbb T^d)}.
\end{equation}
Besides, Young's inequality for convolution indicates that
\begin{equation}
	\| P_{\leq N} u(t,\cdot)\|_{L^\infty_{x} (\mathbb T^d)} \le \| K_N (t,\cdot) \|_{L^\infty_{x} (\mathbb{T}^d)} \| u_0 \|_{L^1_{x} (\mathbb{T}^d)}.
\end{equation}
Interpolating with these two bounds, we have that for any $r\geq 2$,
\begin{equation} \label{est:kernel-Lr}
	\| P_{\leq N} u(t,\cdot) \|_{L^r_{x} (\mathbb{T}^d)} \le \| K_N (t,\cdot) \|_{L^\infty_{x} (\mathbb{T}^d)} ^{1-2/r} \| u_0 \|_{L^{r'}_{x} (\mathbb{T}^d)}.
\end{equation}
We recall the Dirichlet lemma in Diophantine approximation.
\begin{lem}(Dirichlet) \label{L:Dirichlet}
	Given an integer $N\geq 2$ and $\beta\in[0,1]$, there exist integers $1\leq q<N$ and $0\leq a\leq q$ so that $\gcd(a,q)=1$ and
	$
		| \beta - \tfrac aq | \leq \tfrac{1}{Nq}.
	$
\end{lem}
As a consequence, the following dispersive estimate for the Schr\"{o}dinger propagator on the torus holds true.
\begin{lem}~\cite{Bourgain1993} \label{lem:Weyl's bound}
	Let $1\leq q<N$ and $0\leq a\leq q$ so that $\gcd(a,q)=1$ and
	$
		| t - \tfrac aq | \leq \tfrac{1}{Nq}.
	$
	Then we have
	\begin{equation}
		|K^\ominus_N(t,x)| \lesssim \frac{N}{\sqrt{q}(1+N|t-a/q|^{1/2})}.
	\end{equation}
\end{lem}
For the $d$-dimensional case, applying the arithmetic–geometric
mean inequality, one easily gets that for each $t$ there exists $0\leq a_j\leq q_j<N$ so that $\gcd(a_j,q_j)=1$ and
\begin{align}
	| K_N(t,x) | & \lesssim \prod_{j=1}^d \frac{N}{\sqrt{q_j}(1+N|\theta_jt-a_j/q_j|^{1/2})} \\& \lesssim \sum_{j=1}^d \Big(\frac{N}{ \sqrt{q_j}(1+N|\theta_jt-a_j/q_j|^{1/2}) }\Big)^d.
\end{align}
Take $\varphi$ to be a suitable cut-off function such that for dyadic $Q,T \in 2^{\mathbb{Z}}$, \begin{equation}
	1 \leq \sum_{Q=1}^{Q_{\text{max}}} \sum_{T=1/N^2}^{T_{\text{max}}} \bigg(\sum_{\substack{ \gcd(a,q)=1 \\ q\sim Q}} \varphi \Big( \frac{t-a/q}{T} \Big)\bigg) + \rho(t),
\end{equation}
where $Q_{\text{max}} < N$ (we will take $Q_{\text{max}}=\sigma N$ for a small constant $\sigma$, so that the major arcs have the separation property, same as~\cite{Bourgain1993}), $T_{\text{max}} = Q_{\text{max}}/(QN^2)$ and $\rho$ is a bounded positive function. As a result, we have
\begin{align}
	\| K_N(t,\cdot)  \|_{L^\infty_{x}(\mathbb T^d)}^{1-2/r} & \lesssim \sum_{j=1}^d \bigg( \rho(\theta_jt) \cdot \Big(\frac{N^2}{Q_{\text{max}}}\Big)^{d(\frac 12-\frac 1r)}  + \sum_{Q,T}\Lambda_{Q,T}(\theta_jt) \cdot (QT)^{-d(\frac 12-\frac 1r)} \bigg) \\&=:\sum_{j=1}^d \Lambda_N(\theta_jt) .\label{est-Lr-lambda}
\end{align}
for \begin{equation}
	\Lambda_{Q,T}(t) = \sum_{\substack{ \gcd(a,q)=1 \\ q\sim Q}} \varphi\bigg(\frac{t-a/q}{T}\bigg).
\end{equation}
We will show time convolution estimates with kernel $\Lambda_N(t)$.
Since irreducible fractions with denominator $\sim Q$ are $Q^{-2} > T$-separated, it holds that $\| \Lambda_{Q,T} \|_{L^\infty} \lesssim 1$ and consequently
$
	\| \Lambda_{Q,T} * g \|_{L^\infty} \lesssim \| g \|_{L^1}.
$
On the other hand, we have the trivial control
\begin{equation}\label{est:time cut-off}
	\Lambda_{Q,T}(t) \leq \Theta_{Q,T}(t) := \sum_{q \sim Q} \sum_{a=0}^{q-1} \varphi\Big(\frac{t-a/q}{T}\Big),
\end{equation}
and the Fourier transform of $\Theta_{Q,T}$ is easier to estimate\begin{align}
	| \hat \Theta_{Q,T} (n) | & =T |\hat\varphi(Tn)| \sum_{q\sim Q}\sum_{a=0}^q \me^{2\pi\mi an/q} =T |\hat\varphi(Tn)| \sum_{q\sim Q}q\chi_{\{q|n\}} \\&\lesssim Q T d(n;Q) | \hat\varphi(T n) |,\quad \forall n\in\mathbb{Z},\label{est:time cut-off fourier}
\end{align}
recall that  $d(n;Q)$ denotes the number of divisors $q$ of $n$ that obeys $q\sim Q$.
\begin{lem}[{\cite[Lemma 4.28]{bour1989israel}}]\label{lem:distribution}
	Denote $d(\ell)$ the number of divisors of $\ell\in\mathbb{Z}$.
	For any integers $B,R\geq 1$ we have
	\begin{equation}
		\# \{1\leq n \leq R : d(n;Q)> D \} \leq  R D^{-B} \sum_{\ell=1}^{Q^B} \frac{d(\ell)^B}\ell.\end{equation}
\end{lem}
\begin{proof}
	By Chebyshev's inequality,
	\begin{align}
		       & \#\{1\leq n \leq R : d(n;Q)> D \} \leq D^{-B} \sum_{n=1}^R (d(n;Q))^B = D^{-B} \sum_{n=1}^R \Big( \sum_{q \sim Q} \chi_{ \{ q | n\} } \Big)^B                                 \\
		={}    & D^{-B} \sum_{n=1}^R \sum_{q_1,\dots, q_B\sim  Q}  \chi_{\{q_i | n,\ 1\leq i \leq B\}} \lesssim D^{-B} \sum_{q_1,\dots,q_B \sim Q} \frac{R}{\operatorname{lcm}(q_1,\dots,q_B)} \\
		={}    & D^{-B} \sum_{\ell=1}^{Q^B} \frac R\ell \#\{q_1,\dots,q_B \sim Q :  \operatorname{lcm}(q_1,\dots,q_B) = \ell\}                                                                 \\
		\leq{} & D^{-B} \sum_{\ell=1}^{Q^B} \frac R\ell \#\{q_1,\dots,q_B \sim Q : q_i | \ell,1\leq i\leq B\} = D^{-B} \sum_{\ell=1}^{Q^B} \frac R\ell \#\{q : q|\ell
		\}^B.
	\end{align}
\end{proof}

\begin{lem}\label{lem:interpolated time convolution estimate}
	For $0<T<1$ and $1\leq Q \leq T^{-1/2}$, we have that for any $p>2$ and $f,g\in L^{p'}([0,1])$, we have
	\begin{align} \label{convolution-bound}
		| \langle f, g*\Lambda_{Q,T}\rangle|  \lesssim_p {} Q^2T \|f\|_{L^1}\|g\|_{L^1} + Q^{\frac{2 (1+\varepsilon)}p} T^{\frac 2p} \| f\|_{L^{p'} } \| g \|_{L^{p'}}.
	\end{align}
\end{lem}
\begin{proof}
	We start from the special case $f=\chi_E,g=\chi_F$ for $E,F\subset[0,1]$ and some $p_0\in (2,p)$. Using \eqref{est:time cut-off} \eqref{est:time cut-off fourier} and taking Fourier transform, the LHS is bounded by
	\begin{equation}
		Q T \sum_{n \in \mathbb Z}  |\hat\chi_E(n)| |\hat\chi_F(n)| |\hat\varphi(T n )| d(n;Q).
	\end{equation}
	From Lemma~\ref{lem:distribution}, the distribution of $d(n;Q)$ \begin{equation}
		\# \{1\leq n \leq 2^kT^{-1} : d(n;Q)> D \} \leq 2^kA(D/Q)^{-B}  \ \text{ with }\ A=T^{-1} Q^{-B} \sum_{\ell=1}^{Q^B} \frac{d(\ell)^B}\ell.\end{equation}
	Separating the summation by $d(n;Q) \leq D$ and $d(n;Q)>D$ to get that
	\begin{multline}
		Q T \bigg[ Q |\hat\chi_E(0)| | \hat\chi_F(0) | + D \sum_{d(n;Q)\leq D } |\hat\chi_E(n)| | \hat\chi_F(n)|  \\ + Q \sum_{k=0}^\infty \sum_{\substack{n \sim 2^k T^{-1} \\ d(n;Q)>D}} |\hat\chi_E(n)| | \hat\chi_F(n)| |\hat\varphi(T n )| \bigg] ,
	\end{multline}
	which is further bounded by
	\begin{equation}
		Q^2 T\Big[ |E| |F| + D/Q |E|^{\frac12}|F|^{\frac12} +  (D/Q)^{-\frac Bq} A^{\frac1q} |E|^{\frac12+\frac1{2q}} |F|^{\frac12+\frac1{2q}} \Big],
	\end{equation}
	for any $q \in [1,\infty)$ since $|\hat\varphi|$ rapidly decays. Take $D = Q ( A^2|E| |F| )^{ \frac1{2(B+q)} }$ to get the bound
	\begin{align}
		    & Q^2T |E| |F| +  Q^2 T A^{\frac 1{B+q}} (|E| |F|)^{ \frac 12  (1 + \frac1{B+q}) }                                                                                   \\
		={} & Q^2T |E| |F| +  T^{1-\frac 1{B+q}} Q^{2-\frac B{B+q}}  ( |E| |F| )^{\frac 12 (1+\frac 1{B+q})} \Big(\sum_{\ell=1}^{Q^B} \frac{d(\ell)^B}{\ell}\Big)^{\frac1{B+q}}.
	\end{align}
	Recall that $d(\ell) = \exp(O(\log \ell/\log\log\ell))$,
	we get the desired bound \eqref{convolution-bound} for
	\begin{equation}
		p_0= \frac2{1-1/(B+q)},\quad \varepsilon > \frac{p_0}2 \Big(2-\frac B{B+q}\Big) -1.
	\end{equation}

	For general $f,g$ and $p>p_0$ we apply Lemma \ref{lem:atomic decomposition} below  to write
	\begin{equation}\label{fg-decomposition}
		f = \sum_kf_k \chi_{E_k},\ g = \sum_\ell g_\ell \chi_{F_\ell}
	\end{equation}
	with
	\begin{align}
		 & |E_k| \leq 2^{-k},       &  & |f_k\chi_{E_k}|\leq \lambda_k2^{k/p'},         &  & \sum_k \lambda_k^q 2^{k(\frac q{p'}-1)} \lesssim \| f\|_{L^q}^q, \label{decomposition-fp}      \\
		 & |F_\ell| \leq 2^{-\ell}, &  & |g_\ell\chi_{F_\ell}|\leq \mu_\ell2^{\ell/p'}, &  & \sum_\ell \mu_\ell^q 2^{\ell(\frac q{p'}-1)} \lesssim \| g\|_{L^q}^q. \label{decomposition-gp}
	\end{align}
	Then
	\begin{equation}
		\langle f,g*\Lambda_{Q,T} \rangle = \sum_{k+\ell<j_0} \langle f_k,g_\ell*\Lambda_{Q,T} \rangle + \sum_{k+\ell\geq j_0 } \langle f_k,g_\ell*\Lambda_{Q,T} \rangle. \label{sum1+2}
	\end{equation}
	For the second term in \eqref{sum1+2}, we use $L^\infty$ bound on $\Lambda_{Q,T}$ and $L^1$ on $f_k, g_l$ to obtain
	\begin{equation}
		\sum_{k+\ell\geq j_0} |\langle f_k,g_\ell*\Lambda_{Q,T} \rangle | \lesssim \sum_{k+\ell\geq j_0} \lambda_k \mu_\ell 2^{ -\frac{k+\ell}{p} } \lesssim 2^{-\frac{j_0}p}\| \lambda_k \|_{\ell^{p'}} \| \mu_\ell \|_{\ell^{p'} }. \label{bound-kl-large}
	\end{equation}
	For the first term in  \eqref{sum1+2}, we apply \eqref{convolution-bound} with exponent $p_0$ for characteristic functions and get
	\begin{align}
		\sum_{k+\ell<j_0} |\langle f_k,g_\ell*\Lambda_{Q,T}| \rangle & \lesssim \sum_{k+\ell<j_0} \lambda_k\mu_\ell2^{\frac{k+\ell}{p'}} \Big[ Q^2T 2^{-(k+\ell)} + Q^{\frac{2(1+\varepsilon)}{p_0}}T^{\frac2{p_0}} 2^{-\frac{k+\ell}{p_0'}}\Big] \\&
		\lesssim Q^2T \sum_{k+\ell<j_0} \lambda_k \mu_\ell 2^{-\frac{k+\ell}{p}} + Q^{\frac{2(1+\varepsilon)}{p_0}}T^{\frac2{p_0}} \sum_{k+\ell<j_0} \lambda_k\mu_\ell 2^{\frac{k+\ell}{p'}-\frac{k+\ell}{p_0'}}                                  \\&
		\lesssim Q^2T \| f \|_{L^1} \| g \|_{L^1} + Q^{\frac{2(1+\varepsilon)}{p_0}}T^{\frac2{p_0}} 2^{j_0(\frac1{p_0}-\frac 1p)} \| \lambda_k\|_{p'} \| \mu_\ell \|_{p'}. \label{bound-kl-small}
	\end{align}
	Choosing $j_0$ to be the smallest positive integer satisfying
	\begin{equation}
		Q^{\frac{2(1+\varepsilon)}{p_0}}T^{\frac2{p_0}} 2^{j(\frac1{p_0}-\frac 1p)} \geq 2^{-\frac jp},
	\end{equation}
	and combining \eqref{bound-kl-large} and \eqref{bound-kl-small} with the chosen $j_0$,
	we get the desired bound \eqref{convolution-bound}.
\end{proof}
\begin{lem}[Atomic decomposition]\label{lem:atomic decomposition}
	For any $f\in L^p([0,1])$, there exists a sequence of sets $\{E_k\}$ and positive numbers $\{\lambda_k\}$ and functions $\{f_k\}$ such that $f=\sum f_k\chi_{E_k}$ and
	\begin{equation}
		|E_k| \leq 2^{-k},\quad |f_k\chi_{E_k}|\leq \lambda_k 2^{k/p}, \quad \sum_k \lambda_k^q 2^{k(\frac qp-1)} \lesssim \| f\|_{L^q}^q,\quad \forall q\geq1.
	\end{equation}
\end{lem}
\begin{proof}
	We set
	\begin{equation}
		\eta_k = \inf\{ \eta>0: |\{ |f|>\eta \}|<2^{-k} \},\text{ and } E_k = \{ \eta_k<|f|\leq\eta_{k+1} \}
	\end{equation}
	Then
	\begin{equation}
		\sum_k \eta_k \chi_{E_k} \leq |f| \leq \sum_{k} \eta_{k+1} \chi_{E_{k}},
	\end{equation}
	hence
	\begin{align}
		\| f \|_{L^q}^q \leq \sum_k \eta^q_{k+1}  2^{-k}.
	\end{align}
	On the other hand, by definition we know $\eta<\eta_{k}$ implies $|\{|f|>\eta\}| \geq 2^{-k}$,
	\begin{align}
		\sum_k \eta^q_{k+1}2^{-k} = q\int_0^{\infty}  \eta^{q-1}\Big (\sum_k 2^{-k} \mathbf1_{\eta<\eta_{k+1}}\Big)\dif \eta \lesssim q\int_0^\infty \eta^{q-1} |\{|f|>\eta\}| \dif \eta.
	\end{align}
	It suffices to take $f_k = f\chi_{E_k}$ and $\lambda_k = \eta_{k+1}2^{-k/p}$.
\end{proof}

As a consequence we obtain the convolution estimate with kernel $\Lambda_N(t)$, while \cite[Lemma 2.8]{Killip2016} ensures the same convolution estimate for $\Lambda_N(\theta_j t)$.
\begin{lem}\label{lem:full time convolution estimate}
	Let $r,p>2$ such that
	\begin{equation}
		\frac 2p< d\Big(\frac 12-\frac 1r\Big)<2.\end{equation}
	Then for any $f,g\in L^{p'}([0,1])$, we have
	\begin{equation}
		|\langle f,g*\Lambda_N \rangle| \lesssim \Big(\frac{N^2}{Q_{\text{max}}}\Big)^{d(\frac 12-\frac1r)} \|f\|_{L^1}\|g\|_{L^1} + N^{2(d(\frac 12-\frac1r)-\frac2p)} \| f\|_{L^{p'} } \| g \|_{L^{p'}}.
	\end{equation}
\end{lem}
\begin{proof}
	Form the definition of $\Lambda_N$ and Lemma \ref{lem:interpolated time convolution estimate} we see that $  |\langle f,g*\Lambda_N \rangle|$ is controlled by
	\begin{align}
		\sum_{Q=1}^{Q_{\text{max}}} \sum_{T=1/N^{2}}^{T_{\text{max}}} & (QT)^{-d(\frac12-\frac1r)} | \langle f,g*\Lambda_{Q,T} \rangle| + \Big( \frac{N^2}{Q_{\text{max}}} \Big)^{d(\frac 12-\frac 1r)} | \langle f,g*\rho \rangle| \\
		\lesssim {}
		\sum_{Q=1}^{Q_{\text{max}}} \sum_{T=1/N^{2}}^{T_{\text{max}}} & (QT)^{-d(\frac12-\frac1r)}  \Big[Q^2T \|f\|_{L^1}\|g\|_{L^1} +  Q^{\frac{2 (1+\varepsilon)}p} T^{\frac 2p} \| f\|_{L^{p'} } \| g \|_{L^{p'}}\Big]           \\
		                                                              & + \Big(\frac{N^2}{Q_{\text{max}}}\Big)^{d(\frac 12-\frac1r)} \|f\|_{L^1}\|g\|_{L^1}.
	\end{align}
	Clearly, we have the summations
	\begin{align}
		\sum_{Q,T} (QT)^{-d(\frac12-\frac1r)} Q^2T & = \Big(\frac {N^2}{Q_{\text{max}}}\Big)^{d(\frac 12-\frac1r)} \sum_{Q,T} \Big(\frac Q{Q_{\text{max}}}\Big)^{2-d(\frac 12-\frac1r)} (N^2T)^{1-d(\frac 12-\frac1r)} \\&\lesssim \Big(\frac{N^2}{Q_{\text{max}}}\Big)^{d(\frac 12-\frac1r)},
	\end{align} and
	\begin{align}
		\sum_{Q,T} (QT)^{-d(\frac12-\frac1r)} Q^{\frac{2 (1+\varepsilon)}p} T^{\frac 2p} & = \sum_Q Q^{-d(\frac 12-\frac1r) + \frac{2(1+\varepsilon)}{p}} \sum_{T} T^{-d(\frac 12-\frac1r) + \frac 2p} \\&\lesssim N^{2(d(\frac 12-\frac 1r)-\frac 2p)} ,
	\end{align}
	as long as
	$\varepsilon$ is sufficiently small.
	This completes the proof of the lemma.
\end{proof}

Our final step of proving sharp Strichartz estimate is the standard $TT^*$ argument.
\begin{proof}[Proof of Theorem~\ref{thm:main sharp Strichartz estimate on d-dimensional tori}]
	For any $u_0\in L^2_x(\mathbb{T}^2)$,
	\begin{align}
		\| P_{\leq N} \me^{\mi t\Box } u_0 \|_{{L_{t,x}^{p}([0,1]\times \mathbb T^{d})}} & = \sup_{\substack{ \| F \|_{ {L_{t,x}^{{p'}}([0,1]\times \mathbb T^{d})} } \leq 1 \\F\neq0}} |\langle P_{\leq N} \me^{\mi t\Box} u_0, F\rangle | \\& \leq \sup_F \| u_0\|_{L^2}| \langle F,F*K_N\rangle|^{1/2} \label{Strichartz-p1}\\&
		\leq \| u_0 \|_{L^2} \sup_F \| F*K_N \|_{{L_{t,x}^{p}([0,1]\times \mathbb T^{d})}}^{1/2},
	\end{align}
	by taking $Q_{\text{max}} = \sigma N$ for $\sigma$ small enough and $r=p=\frac{2(d+4)}{d}$ in Lemma \ref{lem:full time convolution estimate}\begin{align}
		\| F*K_N \|_{{L_{t,x}^{p}([0,1]\times \mathbb T^{d})}} & =\bigg\| \iint K_N (t',x') F (t-t',x-x')  \dif x' \dif t' \bigg\|_{{L_{t,x}^{p}([0,1]\times \mathbb T^{d})}} \\&
		\leq \bigg\| \int \bigg\| \int K_N (t',x') F(t-t',x-x') \dif x'\bigg\|_{L^p_{x} (\mathbb{T}^d)} \dif t' \bigg\|_{L^p_t([0,1])}                                        \\&
		\lesssim \sum_{j=1}^d \bigg\| \int \Lambda_N(\theta_jt') \| F(t-t',\cdot) \|_{L^{p'}_{x} (\mathbb{T}^d)} \dif t' \bigg\|_{L^p_t([0,1])} \label{Strichartz-p2}         \\&
		\lesssim \Big( N^{d(\frac 12 - \frac 1p) }+ N^{2(d(\frac12-\frac 1p)-\frac2p)} \Big) \| F \|_{{L_{t,x}^{p'}([0,1]\times \mathbb T^{d})}}.\label{Strichartz-proof}
	\end{align}
	We applied \eqref{est:kernel-Lr} and \eqref{est-Lr-lambda} to pass from the second to the third line. The step from the third to the last line follows from Lemma \ref{lem:full time convolution estimate} and a duality argument, together with the trivial inequality $\| F \|_{L_t^1([0,1])} \leq \| F \|_{L_t^{p'}([0,1])}$.
\end{proof}
\begin{remark}
	It is clear that only in the last step, the power $p=\frac{2(d+4)}{d}$ becomes special. It is exactly the power that balances the contribution from two terms in \eqref{Strichartz-proof}.
\end{remark}

\section{Bilinear Estimates}\label{Sec:bilinear}
In the remainder of this paper, we focus on the two-dimensional square torus $\mathbb{T}^2$.
\subsection{Line concentrated functions}As shown in \cite{wang2013},   the main obstruction to improving the Strichartz estimate for the hyperbolic Schr\"{o}dinger equation is the contributions coming from those frequencies localized on the diagonal of $\mathbb Z^2$, resulting worse $L^p$ estimate  than elliptic Schr\"{o}dinger for small exponent $p$. Whereas, the diagonal part admits much smaller Fourier support, which reduces the loss in $L^p$ estimate for large exponent $p$.
\begin{lem}
	Let $\ell \subset \mathbb{Z}^2$ be a line. Suppose $\phi_1,\phi_2\in L^2_x(\mathbb{T}^2)$ such that the Fourier support of $\phi_1\cdot\phi_2$ is contained in $\ell\cap C$, with $C$ being a square of sidelength $2N$. \begin{enumerate}
		\item We have the estimate
		      \begin{equation}\label{estv2:Bernstein estimate for diagonal part}
			      \| \phi_1\cdot \phi_2 \|_{L^\infty_x(\mathbb{T}^2)} \lesssim N \| \phi_1 \|_{L^2_x(\mathbb{T}^2)} \| \phi_2 \|_{L^2_x(\mathbb{T}^2)}.
		      \end{equation}
		\item For any $\phi\in L^2_x(\mathbb{T}^2)$ and $M\geq 1$,
		      \begin{equation}\label{estv2:trilinear diagonal orthogonality}
			      \| \phi_1\phi_2 P_{\leq M}\phi \|_{L^2_x(\mathbb{T}^2)} \lesssim N^{1/2}M^{1/2} \| \phi_1 \|_{L^2_x(\mathbb{T}^2)} \| \phi_2 \|_{L^2_x(\mathbb{T}^2)} \| \phi \|_{L^2_x(\mathbb{T}^2)}.
		      \end{equation}
	\end{enumerate}
	\begin{proof}
		Since $\widehat{\phi_1\phi_2}=\hat\phi_1*\hat\phi_2$, by Young's inequality for convolution it holds that $\| \widehat{\phi_1\phi_2}\|_{\ell^\infty} \leq \| \hat\phi_1\|_{\ell^2}\| \hat\phi_2\|_{\ell^2}$, hence
		\begin{align}
			\| \phi_1\cdot\phi_2\|_{L^\infty_x(\mathbb{T}^2)} &
			\leq \| \widehat{\phi_1\phi_2}\|_{\ell^1} \leq
			\#(\ell\cap C)
			\| \hat\phi_1\|_{\ell^2}\| \hat\phi_2\|_{\ell^2}    \\& \lesssim N \| \phi_1 \|_{L^2_x(\mathbb{T}^2)} \| \phi_2 \|_{L^2_x(\mathbb{T}^2)}.
		\end{align}
		As for the second part of the Lemma, notice that the set $\{\ell'\}$ of all lines parallel with $\ell$ is a disjoint partition of $\mathbb{Z}^2$. Decomposing
		\begin{equation}
			\phi_1\phi_2P_{\leq M}\phi = \sum_{\ell'} \phi_1\phi_2P_{\ell'}P_{\leq M}\phi,
		\end{equation}
		we see that their Fourier support are disjoint with each other, by orthogonality
		\begin{align}
			\| \phi_1 \phi_2 P_{\leq M} \phi\|_{{L^2_{x}(\mathbb T^{2})}}^2 & = \sum_{\ell'} \| \phi_1\phi_2P_{\ell'}P_{\leq M}\phi\|_{{L^2_{x}(\mathbb T^{2})}}^2 \\&\leq \sum_{\ell'} \| \phi_1\phi_2\|_{{L^2_{x}(\mathbb T^{2})}}^2 \| P_{\ell'}P_{\leq M} \phi \|_{L^\infty_x(\mathbb{T}^2)}^2.
		\end{align}
		Bernstein's inequality indicates that $\| P_{\ell'}P_{\leq M} \phi \|_{L^\infty_x(\mathbb{T}^2)} \lesssim M^{1/2} \| P_{\ell'}P_{\leq M}\phi\|_{{L^2_{x}(\mathbb T^{2})}}$, while
		\begin{align}
			\| \phi_1\cdot\phi_2\|_{L^2_x(\mathbb{T}^2)}^2 & = \| \widehat{\phi_1\phi_2}\|_{\ell^2}^2 \leq
			\#(\ell\cap C)
			\| \widehat{\phi_1\phi_2}\|_{\ell^\infty}^2                                                    \\& \lesssim N \|\phi_1\|_{{L^2_{x}(\mathbb T^{2})}}^2\|\phi_2\|_{{L^2_{x}(\mathbb T^{2})}}^2.
		\end{align}
		Thus
		\begin{align}
			\| \phi_1\phi_1 P_{\leq M}\phi \|_{L^2_x(\mathbb{T}^2)} & = \bigg( \sum_{\ell'} \| \phi_1\phi_2P_{\ell'}P_{\leq M}\phi\|_{{L^2_{x}(\mathbb T^{2})}}^2 \bigg)^{1/2} \\& \lesssim N^{1/2}M^{1/2} \| \phi_1 \|_{L^2_x(\mathbb{T}^2)} \| \phi_2 \|_{L^2_x(\mathbb{T}^2)} \bigg( \sum_{\ell'} \| P_{\ell'}P_{\leq M} \phi \|_{{L^2_{x}(\mathbb T^{2})}}^2 \bigg)^{1/2}\\&  = N^{1/2}M^{1/2} \| \phi_1 \|_{L^2_x(\mathbb{T}^2)} \| \phi_2 \|_{L^2_x(\mathbb{T}^2)} \| \phi \|_{L^2_x(\mathbb{T}^2)}.
		\end{align}
	\end{proof}
\end{lem}
We will use \eqref{estv2:Bernstein estimate for diagonal part} to prove tail estimate of super-level sets in the next subsection, and \eqref{estv2:trilinear diagonal orthogonality} will be crucial in the proof of Proposition~\ref{est: nonlinear Duhamel terms}.
\subsection{Off-diagonal estimates}
\begin{defi}\label{def:off-diagonal}Suppose $\mathscr{J}$ is a bilinear or sesquilinear operator, we define the ``off-diagonal'' part of $\mathscr{J}$ by
	\begin{equation}\label{def:off-J}
		\widetilde{\mathscr{J}}(u_1,u_2) = \mathscr{J}(u_1,u_2) - \sum_{s\in\mathbb{Z}} \big(\mathscr{J}(P_{\ell_+^s}u_1,P_{\ell_+^s}u_2)+\mathscr{J}(P_{\ell_-^{s}}u_1,P_{\ell_-^s}u_2)\big),
	\end{equation}
	where
	\(\ell_{\pm}^{s} = \{ \xi \in \mathbb{Z}^2 \mid \xi \cdot (1,\pm1) = s \}. \)
\end{defi}

This definition is inspired by the operator $\mathscr{J}_{\mathrm{DS}}$ \eqref{def:JDS}, which admits an off-diagonal nature. Indeed, if we denote $\sqrt{\mathscr{J}_{\mathrm{DS}}}$ the  Fourier multiplier operator with symbol $\sqrt{|m_{\mathrm{DS}}|}$, see~\eqref{def:JDS}.  Then $\widetilde{\mathscr{J}}_{\mathrm{DS}}=\mathscr{J}_{\mathrm{DS}}$ and  $\widetilde{\sqrt{\mathscr{J}_{\mathrm{DS}}}}=\sqrt{\mathscr{J}_{\mathrm{DS}}}$.
It turns out that the off-diagonal part shares similarities with elliptic NLS.

For simplicity, we make the following convention throughout this subsection.
\begin{trivlist}
	\item[\hskip \labelsep {\bf Convention.}]
	Let $\phi_1,\phi_2 \in L^2_x(\mathbb{T}^2)$ satisfying $\| \phi_1 \|_{L^2_x(\mathbb{T}^2)} = \| \phi_2 \|_{L^2_x(\mathbb{T}^2)}=1$, and $u_1 = P_{C_1} \me^{\mi t\Box}\phi_1, u_2=P_{C_2}\me^{\mi t\Box}\phi_2$.  Here $C_1, C_2\subset \mathbb Z^2$ are squares of sidelength $2N$.
\end{trivlist}

We have the tail estimate of super-level sets for linear waves.
\begin{lem}\label{lemv2:level set estimate for linear wave}
	Suppose $\phi\in L^2_x(\mathbb{T}^2)$, $q>4$ and $\lambda\gg N^{1/2}$, $C\subset \mathbb Z^2$ is a square of sidelength $2N$,  then
	\begin{equation}
		\label{estv2:level set estimate for linear wave}
		|\{ (t,x)\in[0,1]\times\mathbb{T}^2 : |P_{C}\me^{\mi t\Box}\phi|>\lambda \}| \lesssim \lambda^{-q} N^{q-4} .\end{equation}
\end{lem}
Estimates of this type first appeared in \cite{Bourgain1993} for the elliptic Schr\"{o}dinger equation (for $d=2$, it is exactly the same bound as \eqref{estv2:level set estimate for linear wave}), see also~\cite{Killip2016}.  \cite{Başakoğlu-Wang2025-3} proved it for the hyperbolic Schr\"{o}dinger equation in the range $p>6$.   We reproduce the proof for completeness.
\begin{proof} Denote
	\[\Omega=\{ (t,x)\in[0,1]\times\mathbb{T}^2 : |P_{C}\me^{\mi t\Box}\phi|>\lambda\},\]
	\[\Omega_w=\{ (t,x)\in[0,1]\times\mathbb{T}^2 : \operatorname{Re} (\me^{\mi w}P_{C}\me^{\mi t\Box}\phi)>\lambda/2\},\]
	with $w\in \{\frac{k}{2}\pi, k=0,1,2,3\}$. Then $|\Omega|\leq 4|\Omega_w|$.   We denote $F=\chi_{\Omega_w}$.

	Let us start with the case $C=[-N,N]^2$. We have
	\begin{equation}\label{omega-est}
		\lambda^2 |\Omega_w|^2\lesssim  |\langle P_{\leq N} \me^{\mi t\Box} \phi, F\rangle |^2
		\leq \|\phi\|_{L^2}^2 \|K_N*F\|_{L^2_{t,x}}^2
		\lesssim  \|F\|_{L^{q'}_{t,x}}\|K_N*F\|_{L^q_{t,x}}.
	\end{equation}
	Then using \eqref{Strichartz-p2} for any $q>4$, and Lemma~\ref{lem:full time convolution estimate} with duality argument,  we get
	\begin{equation}\label{omega-est2}
		\|K_N*F\|_{L^q_{t,x}} \lesssim   N^{1 - \frac 2q }\| F \|_{{L_{t,x}^{1}([0,1]\times \mathbb T^{d})}}+ N^{2(1-\frac4q)}  \| F \|_{{L_{t,x}^{q'}([0,1]\times \mathbb T^{d})}}.
	\end{equation}
	We thus arrive at the estimate
	\begin{equation}\label{omega-est3}
		\lambda^2 |\Omega_w|^2\lesssim N^{1 - \frac 2q}|\Omega_w|^{1+\frac{1}{q'}}+ N^{2(1-\frac4q)}   |\Omega_w|^{\frac{2}{q'}},
	\end{equation}
	which immediately implies  $|\Omega_w|\lesssim \lambda^{-q}N^{q-4}$  by noticing $\lambda \gg N^{\frac12}$.

	Now consider the  general square $C$ with center $r_0\in \mathbb{R}^2$. Then $C_0= C-r_0=[-N,N]^2$. Setting  $\phi_0=\me^{-2\pi \mi r_0\cdot x}\phi(x)$, we get $P_C\phi=  \me^{2\pi \mi r_0\cdot x}P_{C_0}\phi_0$.
	By direct computation, which corresponds to the Galilean transformation for hyperbolic Schr\"{o}dinger equation~\cite{Başakoğlu-Wang2025-3},  we get
	\begin{equation}
		(\me^{\mi t\Box}P_C\phi)(t,x)=  \me^{2\pi \mi(x\cdot r_0+th(r_0))} (\me^{\mi t\Box}P_{C_0}\phi_0)(t, x+t \tilde{r}_0)
	\end{equation}
	with $h$ the symbol of $\Box$ and $\tilde{r}_0=\nabla h(r_0)$.

	Repeating \eqref{omega-est},   we get
	\begin{align}\label{omega-est4}
		\lambda^2 |\Omega_w|^2\lesssim & |\langle P_{C} \me^{\mi t\Box} \phi, F\rangle |^2 = |\langle P_{C_0} \me^{\mi t\Box} \phi_0, \tilde{F}\rangle |^2
	\end{align}
	with $\tilde{F}=  \me^{-2\pi \mi(x\cdot r_0+th(r_0))} F(t, x-t\tilde{r}_0)$.  Since $\tilde{F}$ and $F$ have same $L^{q'}$ norm, we get \eqref{omega-est2}\eqref{omega-est3}, leading  to the same conclusion.

\end{proof}

\begin{coro}
	Let $u_1,u_2$ be as in the Convention, $\mathscr{J}(u_1,u_2)= u_1u_2$ or $u_1\bar u_2$,  $q>2$ and $\lambda\gg N$, then
	\begin{equation}
		\label{estv2:level set estimate for large lambda}
		|\{ (t,x)\in[0,1]\times\mathbb{T}^2 : |\widetilde{\mathscr{J}}(u_1,u_2)|>\lambda \}| \lesssim \lambda^{-q} N^{2(q-2)} .
	\end{equation}
\end{coro}
\begin{proof}
	Lemma \eqref{estv2:Bernstein estimate for diagonal part} indicates that
	\begin{align}
		 & |\{ (t,x)\in[0,1]\times\mathbb{T}^2 : |\widetilde{\mathscr{J}}(u_1, u_2)|>\lambda \}| \\  \leq{}& |\{ (t,x)\in[0,1]\times\mathbb{T}^2 : |u_1|| u_2|>\lambda/2 \}| \\ \le{}& |\{  (t,x)\in[0,1]\times\mathbb{T}^2 : |u_1|>\alpha/2 \}| + |\{  (t,x)\in[0,1]\times\mathbb{T}^2 : |u_2|>\lambda/\alpha \}|,
	\end{align}
	the result follows by taking $\alpha=\lambda^{\frac12}$ and applying Lemma \ref{lemv2:level set estimate for linear wave}  for $2q$.
\end{proof}
Our next ingredient is the bilinear $L^2$ estimate for the off-diagonal part.
\begin{lem}
	\label{lemv2: Bilinear L2 estimate for projection on diagonal}
	Let $u_1,u_2$ be as in the Convention, and $\mathscr{J}(u_1,u_2) = u_1u_2$,  $u_1\bar u_2$, $\mathscr{J}_{\mathrm{DS}}(u_1, u_2)$, or $\sqrt{\mathscr{J}_{\mathrm{DS}}}(u_1,u_2)$, then we have
	\begin{align}
		\label{estv2:L2 bilinear estimate for projection on off-diagnoal}
		\|\widetilde{\mathscr{J}}(u_1, u_2)\|_{L^{2}_{t,x}([0,1]\times\mathbb{T}^2)} \lesssim & N^\varepsilon.
	\end{align}
	for any $\varepsilon>0.$
\end{lem}
\begin{proof}
	By splitting $\hat{\phi}$ into  real and imaginary part, and further to positive and negative part,  we can assume $\hat{\phi}\geq 0$.
	We will use  the notation $\chi = \chi_{\mathbb{Z}^2 \setminus \ell_{\pm}^{0}}$,
	$\mathscr{J}_1(u_1, u_2)=u_1\bar{u}_2$, $\mathscr{J}_2(u_1, u_2)=u_1 {u}_2$.
	Denote  $h$ to be the symbol of $\Box$ $\Gamma$ to be the resonant set corresponding to the group $\{\me^{\mi t\Box}\}_{t\in\mathbb R}$
	\[
		\Gamma := \left\{(\xi_1,\xi_2,\xi_3,\xi_4)\in\mathbb{Z}^{2\times4} \Biggm | \sum_{i=1}^4 (-1)^i \xi_i = 0,\ \sum_{i=1}^4 (-1)^i h(\xi_i) = 0 \right\}.
	\]

	\noindent \textbf{Case 1:}   $\mathscr{J}(u_1,u_2)=\mathscr{J}_1(u_1,u_2)$ and $C_1=C_2$.

	By applying Galilean transformation, we further reduce to the case $C_1=C_2=[-N,N]^2$.
	The Fourier transform of $\widetilde{\mathscr{J}}(u_1, u_2)$ is written as
	\begin{equation} \label{identity:Fourier-ubarv}
		\mathcal{F}(\widetilde{\mathscr{J}}(u_1, u_2) )(\xi,\tau) = \sum_{\substack{\xi_1-\xi_2=\xi\notin\ell_{\pm}^{0}\\ h(\xi_1)-h(\xi_2)=\tau}} \hat\phi_1(\xi_1) \overline{\hat\phi_2(\xi_2)} .
	\end{equation}
	Hence by Plancherel theorem we have
	\begin{align}
		    & \| \widetilde{\mathscr{J}}(u_1, u_2) \|_{L^2_{t,x}([0,1]\times\mathbb{T}^2)}^2 =\sum_\Gamma \chi(\xi_1-\xi_2) \hat\phi_1(\xi_1)\overline{\hat\phi_2(\xi_2)} \hat\phi_2(\xi_3)\overline{\hat\phi_1(\xi_4)} \label{symbol-uu} \\
		={} & \sum_{({\tilde\xi},{\tilde\tau})\in \mathbb{Z}^2\times\mathbb{Z}} \chi(\xi_1-\xi_2) \Big( \sum_{\substack{\xi_1,\xi_3\in {\mathcal{A}_{{\tilde\xi},{\tilde\tau}}}                                                           \\\xi_1+\xi_3={\tilde\xi}}} \hat{\phi}_1(\xi_1)\hat{\phi}_2(\xi_3) \Big)\Big( \sum_{\substack{\xi_2,\xi_4\in {\mathcal{A}_{{\tilde\xi},{\tilde\tau}}}\\\xi_2+\xi_4={\tilde\xi}}} \overline{\hat{\phi}_2(\xi_2)\hat{\phi}_1(\xi_4) }\Big)\label{identity:L2ubarv}
	\end{align}
	here  ${\mathcal{A}_{{\tilde\xi},{\tilde\tau}}} = \{ \xi \in \mathbb{Z}^2 \mid h(\xi) + h({\tilde\xi} - \xi) = {\tilde\tau} \} = \{ \xi \in \mathbb{Z}^2 \mid h(\xi - {\tilde\xi}/2) = {\tilde\tau}/2 - h({\tilde\xi}/2) \}$.

	If ${\tilde\tau}/2 - h({\tilde\xi}/2)\neq0$, we see ${\mathcal{A}_{{\tilde\xi},{\tilde\tau}}}$ denotes a hyperbola in $\mathbb{Z}^2$. Notice that for given $z\neq0$, the number of integer solutions $(x,y)$ to the equation $x^2-y^2=(x+y)(x-y)=z$ is the same as the number of divisors of $z$, which is bounded by $\exp(O(\log|z|/\log\log|z|))$.
	We estimate that $\#({\mathcal{A}_{{\tilde\xi},{\tilde\tau}}} \cap [-N,N]^2) \lesssim N^\varepsilon$ and hence by Cauchy-Schwarz inequality
	\begin{align}
		            & \sum_{ {\tilde\tau}/2 - h({\tilde\xi}/2) \neq 0} \Big( \sum_{\substack{\xi_1,\xi_3\in {\mathcal{A}_{{\tilde\xi},{\tilde\tau}}}                                                                                                                              \\\xi_1+\xi_3={\tilde\xi}}} \hat{\phi}_1(\xi_1)\hat{\phi}_2(\xi_3) \Big)\Big( \sum_{\substack{\xi_2,\xi_4\in {\mathcal{A}_{{\tilde\xi},{\tilde\tau}}}\\\xi_2+\xi_4={\tilde\xi}}} \overline{\hat{\phi}_2(\xi_2)\hat{\phi}_1(\xi_4) }\Big) \\
		\lesssim {} & \sum_{{\tilde\tau}/2 - h({\tilde\xi}/2) \neq0} \#({\mathcal{A}_{{\tilde\xi},{\tilde\tau}}}\cap[-N,N]^2) \Big( \sum_{\xi\in {\mathcal{A}_{{\tilde\xi},{\tilde\tau}}}} |\hat{\phi}_1(\xi)|^2 |\hat{\phi}_2({\tilde\xi}-\xi) |^2 \Big) \lesssim N^\varepsilon.
	\end{align}

	If ${\tilde\tau}/2 - h({\tilde\xi}/2)=0$, we see that
	$ {\mathcal{A}_{{\tilde\xi},{\tilde\tau}}} = \ell_+^{s_+} \cup \ell_-^{s_-}$ with $s_\pm = (1,\pm1)\cdot {\tilde\xi}/2$. By the definition of $\widetilde{\mathscr{J}}$, one knows $\xi_1,\xi_2$ cannot lie on the same line.
	In the case $\xi_1 \in \ell_+^{s_+}$, $\xi_2\in \ell_-^{s_-}$, we have ${\tilde\xi} - \xi_1 \in \ell_+^{s_+}$, ${\tilde\xi}-\xi_2 \in \ell_-^{s_-}$, hence
	\begin{align*}
		\sum_{\xi_1 \in \ell_+^{s_+}} |\hat\phi_1(\xi_1)|  | \hat\phi_2({\tilde\xi}-\xi_1) | \leq{} & \Bigg(\sum_{\xi_1 \in \ell_+^{s_+}} |\hat\phi_1(\xi_1)|^2 \Bigg)^{1/2}\Bigg(  \sum_{\xi_1 \in \ell_+^{s_+}} | \hat\phi_2({\tilde\xi}-\xi_1) |^2 \Bigg)^{1/2} \\
		={}                                                                                         & \| P_{\ell^{s_+}_+}\phi_1 \|_{{L^2_{x}(\mathbb T^{2})}}\| P_{\ell^{s_+}_+}\phi_2 \|_{{L^2_{x}(\mathbb T^{2})}}.
	\end{align*}
	The same is true for $\xi_2$ and ${\tilde\xi}-\xi_2$. Hence
	\begin{align}
		 & \sum_{ {\tilde\xi}} \Big( \sum_{\substack{\xi_1,\xi_3\in  \ell_+^{s_+} \\\xi_1+\xi_3={\tilde\xi}}} \hat\phi_1(\xi_1) \hat\phi_2(\xi_3) \Big)\Big( \sum_{\substack{\xi_2,\xi_4\in \ell_-^{s_-}\\\xi_2+\xi_4={\tilde\xi}}} \overline{\hat\phi_2(\xi_2) \hat\phi_1(\xi_4) }\Big) \\ \leq{} &\sum_{s_+\in\mathbb Z} \| P_{\ell^{s_+}_+}\phi_1 \|_{{L^2_{x}(\mathbb T^{2})}} \| P_{\ell^{s_+}_+}\phi_2 \|_{{L^2_{x}(\mathbb T^{2})}} \sum_{\ s_-\in\mathbb Z} \| P_{\ell^{s_-}_-}\phi_1 \|_{{L^2_{x}(\mathbb T^{2})}}\| P_{\ell^{s_-}_-}\phi_2  \|_{{L^2_{x}(\mathbb T^{2})}} \\ \leq{}&1.
	\end{align}
	It's similar for  the case  $\xi_1 \in \ell_-^{s_-}$, $\xi_2 \in \ell_+^{s_+}$.

	\noindent \textbf{Case 2:}  $\mathscr{J}(u_1,u_2)=\mathscr{J}_1(u_1,u_2)$ and $C_1\not=C_2$.

	Notice that
	$(\xi_1,\xi_2,\xi_3,\xi_4)\in \Gamma$ implies \begin{equation}
		0= \sum_{i}(-1)^{i-1}h(\xi_i)= 2(\xi_1-\xi_4)^t A(\xi_1-\xi_2) , \quad A=\begin{bmatrix}
			1 & 0 \\0& -1
		\end{bmatrix}. \label{G-identity}
	\end{equation}
	Hence we get
	\begin{equation} \label{symbol-control}
		\chi_{\mathbb{Z}^2 \setminus \ell_{\pm}^{0}}(\xi_1-\xi_2) \lesssim   \chi_{\mathbb{Z}^2 \setminus \ell_{\pm}^{0}}(\xi_1-\xi_4),\quad (\xi_1,\xi_2,\xi_3,\xi_4)\in \Gamma.
	\end{equation}
	Applying \eqref{symbol-control} to identity \eqref{symbol-uu}, we get
	\begin{align}
		       & \| \widetilde{\mathscr{J}}(u_1, u_2) \|_{L^2_{t,x}([0,1]\times\mathbb{T}^2)}^2
		\lesssim {}  \sum_\Gamma \chi(\xi_1-\xi_4) |\hat\phi_1(\xi_1)\overline{\hat\phi_2(\xi_2)} \hat\phi_2(\xi_3)\overline{\hat\phi_1(\xi_4)}|                                                                                    \\
		={}    &
		\sum_{(a,j)\in\mathbb{Z}^2\times\mathbb Z } |\chi(a)|^2 \Bigg(\sum_{\substack{ \xi_1-\xi_4=a                                                                                                                                \\ h(\xi_1)-h(\xi_4)=j }}  |\hat{\phi}_1(\xi_1)\overline{\hat{\phi}_1(\xi_4)}|\Bigg)\Bigg(\sum_{\substack{ \xi_2-\xi_3=a \\ h(\xi_2)-h(\xi_3)=j }}  |\overline{\hat{\phi}_2(\xi_2)}\hat{\phi}_2(\xi_3)|\Bigg) \\
		\leq{} & \begin{multlined}[t][.95\textwidth]
			         \Bigg( \sum_{(a,j)\in\mathbb{Z}^2\times\mathbb Z } |\chi(a)|^2 \Bigg|\sum_{\substack{ \xi_1-\xi_4=a \\ h(\xi_1)-h(\xi_4)=j }}  |\hat{\phi}_1(\xi_1)\overline{\hat{\phi}_1(\xi_4)}|\Bigg|^2 \Bigg)^{1/2} \times \\ {
			         \Bigg(\sum_{(a,j)\in\mathbb{Z}^2\times\mathbb Z } |\chi(a)|^2 \Bigg|\sum_{\substack{ \xi_2-\xi_3=a \\ h(\xi_2)-h(\xi_3)=j }}  |\overline{\hat{\phi}_2(\xi_2)}\hat{\phi}_2(\xi_3)|\Bigg|^2\Bigg)^{1/2} }
		         \end{multlined} \\
		={}    & \| \widetilde{\mathscr{J}}_1(u_1, u_1) \|_{L^{2}_{t,x}([0,1]\times\mathbb{T}^2)} \| \widetilde{\mathscr{J}}_1(u_2, u_2) \|_{L^{2}_{t,x}([0,1]\times\mathbb{T}^2)} \lesssim N^{2\varepsilon}.
	\end{align}
	Here we used \eqref{identity:Fourier-ubarv} and Case 1 to conclude the estimate.

	\noindent \textbf{Case 3:}   $\mathscr{J}(u_1,u_2)=\mathscr{J}_{\mathrm{DS}}(u_1, u_2) $ or $\sqrt{\mathscr{J}_{\mathrm{DS}}}(u_1, u_2) $.

	We notice
	\begin{equation}
		|m_{\mathrm{DS}}(\xi_1-\xi_2)| \lesssim \chi_{\mathbb{Z}^2 \setminus \ell_{\pm}^{0}}(\xi_1-\xi_2), \quad \sqrt{|m_{\mathrm{DS}}(\xi_1-\xi_2)|} \lesssim \chi_{\mathbb{Z}^2 \setminus \ell_{\pm}^{0}}(\xi_1-\xi_2).
	\end{equation}
	The conclusion follows from Case 2.

	\noindent \textbf{Case 4:}   $\mathscr{J}(u_1,u_2)=\mathscr{J}_2(u_1,u_2)$.

	We start with the Fourier transform
	\begin{equation} \label{identity:Fourier-uv}
		\mathcal{F}(\widetilde{\mathscr{J}}(u_1, u_2) )(\xi,\tau) = \sum_{\substack{\xi_1+\xi_3=\xi \\ h(\xi_1)+h(\xi_3)=\tau}}\chi(\xi_1-\xi_3) \hat\phi_1(\xi_1)  \hat\phi_2(\xi_3) .
	\end{equation}
	Now we get
	\begin{align}
		\| \widetilde{\mathscr{J}}(u_1, u_2) \|_{L^2_{t,x}([0,1]\times\mathbb{T}^2)}^2
		=  \sum_{\Gamma} \chi(\xi_1-\xi_3)\chi(\xi_2-\xi_4)\hat\phi_1(\xi_1) \hat\phi_2(\xi_3) \overline{\hat\phi_1(\xi_4)\hat\phi_2(\xi_2)} \label{symbol-uv}
	\end{align}
	From \eqref{G-identity}, we have
	\begin{equation}
		\chi(\xi_1-\xi_3)\chi(\xi_2-\xi_4)\lesssim \chi(\xi_1-\xi_4)\chi(\xi_2-\xi_3), \quad (\xi_1,\xi_2,\xi_3,\xi_4)\in \Gamma.
	\end{equation}
	Thus we get
	\begin{align}
		\| \widetilde{\mathscr{J}}(u_1, u_2) \|_{L^2_{t,x}([0,1]\times\mathbb{T}^2)}^2
		\lesssim  \sum_{\Gamma} \chi(\xi_1-\xi_4)\chi(\xi_2-\xi_3)|\hat\phi_1(\xi_1) \hat\phi_2(\xi_3) \overline{\hat\phi_1(\xi_4)\hat\phi_2(\xi_2)}|
	\end{align}
	The rest follows identically as in the proof of Case 2, and we obtain the bound
	\begin{align}
		\| \widetilde{\mathscr{J}}(u_1, u_2) \|_{L^2_{t,x}([0,1]\times\mathbb{T}^2)}^2 & \lesssim    \| \widetilde{\mathscr{J}}_1(u_1, u_1) \|_{L^2_{t,x}([0,1]\times\mathbb{T}^2)}  \| \widetilde{\mathscr{J}}_1(u_2, u_2) \|_{L^2_{t,x}([0,1]\times\mathbb{T}^2)} \\&\lesssim N^{2\varepsilon}.
	\end{align}
\end{proof}

Interpolation between \eqref{estv2:level set estimate for large lambda} and \eqref{estv2:L2 bilinear estimate for projection on off-diagnoal} leads to an improved $L^p$ estimate for $\widetilde{\mathscr{J}}(u_1, u_2)$, which is better than \eqref{est:main sharp Strichartz estimate on d-dimensional tori}.  This procedure represents a bilinear analog of the $\varepsilon$
removal argument in ~\cite{Bourgain1993,Killip2016}.
\begin{lem}\label{lem:refinedLp}
	Let $u_1,u_2$ be as in the Convention,  $\widetilde{\mathscr{J}}(u_1, u_2)=u_1u_2$ or $u_1\bar{u}_2$,  then for any $p>2$ we have
	\begin{equation}
		\label{estv2:epsilon improved bilinear estimate for projection on off-diagnoal}
		\|\widetilde{\mathscr{J}}(u_1, u_2)\|_{L^{p}_{t,x}([0,1]\times\mathbb{T}^2)} \lesssim N^{2-4/p}.\end{equation}
\end{lem}

\begin{proof}
	It suffices to show
	\begin{equation}
		\int \lambda^{p-1} |\{ (t,x)\in[0,1]\times\mathbb{T}^2 : |\widetilde{\mathscr{J}}(u_1, u_2)|>\lambda \}| \dif \lambda
		\lesssim N^{2(p-2)}.
	\end{equation}
	For the integral on $N^{1+\sigma} < \lambda \lesssim N^2$ we choose some $2<q<p$ and use \eqref{estv2:level set estimate for large lambda} to estimate that
	\begin{equation}
		\int_{N^{1+\sigma}}^{N^2} \lambda^{p-1} \lambda^{-q} N^{2(q-2)}\dif \lambda \lesssim N^{2(p-2)}.\end{equation}
	For the integral on $\lambda < N^{1+\sigma}$ we use \eqref{estv2:L2 bilinear estimate for projection on off-diagnoal} to estimate that
	\begin{equation}
		\int_0^{N^{1+\sigma}} \lambda^{p-1} \lambda^{-2} N^{2\varepsilon}\dif \lambda \lesssim N^{(1+\sigma)(p-2)}N^{2\varepsilon} .\end{equation}
	We reach the conclusion by taking $\varepsilon$ and $\sigma$ small enough.
\end{proof}

\section{Local Well-posedness}\label{Sec:LWP}

\subsection{Function spaces}\label{subsec:function-space}
We use the adapted function spaces $X^s,Y^s$, whose definitions are based on the $U^p,V^p$ spaces.  They originated from the work of Koch-Tataru~\cite{KT}, and have now become a fairly standard framework for  the study of Cauchy problems at critical regularity.
We will quickly summarize their definitions and basic properties. We refer the readers to \cite{hadac,herr2011} for detailed proofs.

Let $\mathcal H$ be a separable Hilbert space over $\mathbb C$; in this paper, this will be $\mathbb C$ or $H^s(\mathbb T^d)$.
Let $\mathcal Z$ be the set of finite partitions $-\infty <t_0 <t_1 < \cdots < t_K \leq \infty$ of the real line.
\begin{defi}
	Let $1 \leq p < \infty$. For $\{ t_k \}_{k=0}^K \in\mathcal Z$ and $ \{ \phi_k \}_{k=0}^{K-1} \subset \mathcal H$ with $\sum_{k=1}^{K-1} \| \phi_k \|_{\mathcal H}^p =1$, we call a piecewise defined function $a:\mathbb R \to \mathcal H$,
	\[ a(t) = \sum_{k=1}^{K-1} \chi_{[t_k,t_{k+1})} \phi_k \]
	a $U^p$-atom, and we define the atomic space $U^p(\mathbb R, \mathcal H)$ of all functions $u \colon \mathbb R\to\mathcal H$ such that
	\[ u = \sum_j \lambda_j a_j ,\quad \mbox{ with } a_j \mbox{ are } U^p\mbox{-atoms, and  } \{\lambda_j\} \in \ell^1,\]
	with norm
	\[
		\| u\|_{U^p(\mathbb R,\mathcal H)} := \inf \left\{ \sum_j |\lambda_j|  \Biggm| u = \sum_j \lambda_j a_j, \  a_j \text{ are } U^p\text{-atoms}\right\}.
	\]
\end{defi}

\begin{defi}
	Let $1\leq p<\infty$, we define the space $V^p(\mathbb R,\mathcal H)$ of functions $v \colon \mathbb R \to \mathcal H$ such that $\lim_{t\to-\infty} v(t)=0$ and the norm
	\[
		\| v \|_{V^p(\mathbb R, \mathcal H)} := \sup_{\{t_k\}_{k=0}^K \in\mathcal Z} \left( \sum_{k=0}^{K-1} \| v(t_{k+1}) - v(t_k) \|_{\mathcal{H}}^p \right)^{1/p}
	\]
	is finite.
\end{defi}

Corresponding to the linear  flow generated by the group $\{\me^{\mi t\Box}\}_{t\in\mathbb R}$, we define the following.

\begin{defi}
	For $s\in \mathbb R$, we define the space $U^p_\Box H^s$ $(\mbox{resp., } V^p_\Box H^s)$ of functions $u \colon \mathbb R \to H^s(\mathbb T^d)$ such that $ t \mapsto \me^{-\mi t\Box} u(t)$ is in $U^p(\mathbb R, H^s(\mathbb T^d))$ $(\mbox{resp., } V^p(\mathbb R, H^s(\mathbb T^d)))$ with the norms
	\[ \| u \|_{U^p_\Box H^s} := \|  \me^{-\mi t\Box} u  \|_{U^p(\mathbb R, H^s(\mathbb T^d))} ,\quad  \| u\|_{V^p_\Box H^s} := \| \me^{-\mi t\Box} u \|_{V^p(\mathbb R, H^s(\mathbb T^d))}  . \]
\end{defi}
Due to the atomic structure of $U^p$, we can extend bounded operators on $L^2(\mathbb T^d)$ to $U^p_\Box L^2$.
\begin{prop} [{\cite[Proposition 2.19]{hadac}}] \label{extension}
	Let $1\leq p< \infty$ and $T_0 \colon L^2(\mathbb T^d) \times \cdots \times L^2(\mathbb T^d) \to L_{\mathrm{loc}}^1 ( \mathbb R \times \mathbb T^d)$ be a $n$-linear operator. If
	\[ \| T_0( \me^{\mi t\Box} \phi_1, \cdots , \me^{\mi t\Box} \phi_n) \|_{L^p_{t,x}} \leq C_{T_0} \prod_{i=1}^n \| \phi_i \|_{{L^2_x(\mathbb T^d)}}, \]
	then $T_0$ extends to a $n$-linear operator $T$ on $U^p_\Box L^2 \times \cdots \times U_\Box^pL^2$, satisfying
	\[ \| T ( u_1, \cdots , u_n) \|_{L^p_{t,x}} \lesssim C_{T_0} \prod_{i=1}^n \| u_i \|_{U^p_\Box L^2}. \]
\end{prop}
We also have the interpolation property.
\begin{prop}[{\cite{hadac, herr2011}}]\label{prop:interpolation} Suppose $q_1, q_2>2$ and $E$ is a Banach space.
	$T: U^{q_1}\times U^{q_2}\rightarrow E$ is a bounded bilinear estimate, with the bound
	\[\|T(u_1,u_2)\|_E\leq C \|u_1\|_{U^{q_1}}\|u_2\|_{U^{q_2}}\]
	Assume there is constant $C_1\in (0,C]$ and we  have estimate
	\[\|T(u_1,u_2)\|_E\leq C_1 \|u_1\|_{V_{2}}\|u_2\|_{V^{2}}\]
	then $T$ satisfies
	\begin{equation}
		\|T(u_1,u_2)\|_E\lesssim C_1 (1+\ln (C/C_1))^2\|u_1\|_{U_{2}}\|u_2\|_{U^{2}}
	\end{equation}

\end{prop}

Now we define the function spaces we will use for Cauchy problem.
\begin{defi}
	For $s\in\mathbb R$, we define the space $X^s$ of functions $u\colon \mathbb R \to H^s(\mathbb T^d)$ such that for every $\xi \in \mathbb Z^d$ the mapping $ t \mapsto \me^{ -\mi t h(\xi)} \widehat{u(t)}(\xi)$ is in $U^2(\mathbb R, \mathbb C ) $, with the norm
	\[ \|u\|_{X^s} := \left( \sum_{\xi \in \mathbb Z^d} \left<\xi\right>^{2s} \| \me^{ -\mi t h(\xi)} \widehat{u(t)}(\xi) \|_{U^2(\mathbb R, \mathbb C)}^2 \right)^{1/2} .\]
\end{defi}
\begin{defi}
	For $s\in\mathbb R$, we define the space $Y^s$ of functions $u\colon \mathbb R \to H^s(\mathbb T^d)$ such that for every $\xi\in \mathbb Z^d$ the mapping $ t \mapsto \me^{ -\mi t h(\xi)} \widehat{u(t)}(\xi)$ is in $V^2(\mathbb R,  \mathbb C ) $, with the norm
	\[ \|u\|_{Y^s} := \left( \sum_{\xi\in \mathbb Z^d} \left<\xi\right>^{2s} \| \me^{ -\mi t h(\xi)} \widehat{u(t)}(\xi) \|_{V^2(\mathbb R, \mathbb C)}^2 \right)^{1/2} .\]
\end{defi}

\begin{remark}\label{rem:embedding}
	We have the embeddings\begin{equation}
		U^2_\Box H^s \hookrightarrow X^s \hookrightarrow Y^s \hookrightarrow V^2_\Box H^s \hookrightarrow U^q_\Box H^s\hookrightarrow L^\infty H^s,\quad \forall q\in(2,\infty). \label{embedding}\end{equation}
\end{remark}
\begin{remark}\label{remark}
	For $s\in \mathbb R$, and $S_1,S_2$ are disjoint subsets of $\mathbb Z^d$, we have
	\[ \| P_{S_1\cup S_2} u\|_{Y^s} ^2 = \| P_{S_1} u\|_{Y^s}^2 + \|P_{S_2} u\|_{Y^s}^2. \]
\end{remark}
For time interval $I\subset \mathbb R$, we also consider the restriction spaces $X^s(I),Y^s(I)$ with norms
\[
	\| u \|_{X^s(I)} = \inf \{ \| \tilde{u} \|_{X^s} \mid \tilde{u}|_{I}=u \},\quad
	\| u \|_{Y^s(I)} = \inf \{ \| \tilde{u} \|_{Y^s} \mid \tilde{u}|_{I}=u \}.
\]
\begin{prop}[{\cite[Proposition 2.10]{herr2011}}]
	Let $s\in \mathbb{R}$ and $T>0$. For $\phi \in H^s(\mathbb T^d)$, we have $\me^{\mi t\Box} \phi \in X^s([0,T))$ and
	\[ \| \me^{\mi t\Box} \phi \|_{X^s([0,T))} \leq \| \phi \|_{H^s(\mathbb T^d)} .\]
	For $f\in L^1 ([0,T); H^s(\mathbb T^d))$, we have  the estimate for the Duhamel term.
	\begin{multline*}
		\left\| \int_0^t \me^{\mi (t-t')\Box} f(t') \dif t' \right\|_{X^s([0,T))} \leq
		\sup_{\substack{v\in Y^{-s}([0,T)) \\ \|v\|_{Y^{-s}([0,T))}\leq 1}  } \left| \iint_{[0,T)\times \mathbb T^d} f(t,x) \overline{v(t,x)} \dif x\dif t \right|.
	\end{multline*}
\end{prop}

\begin{remark}
	The $X^s([0,T))$ norm of the Duhamel term is also controlled by $ \| f \|_{L^1([0,T);H^s(\mathbb T^d))}$.
\end{remark}

\subsection{Bilinear and trilinear estimate}
Now we
transfer the estimates of free waves in previous sections to functions in $Y^0$ spaces.
\begin{lem}\label{linear-Y} For any square $C$ of sidelength $N$ in $\mathbb{Z}^2$, we have
	\begin{equation}\label{est:linear-Y}
		\| P_{C} u \|_{{L_{t,x}^{p}([0,1]\times \mathbb T^{2})}} \lesssim \big( N^{1-\frac{4}p} + N^{\frac 12-\frac {1}{p}} \big) \|u \|_{Y^0},\quad \forall p>2.
	\end{equation}
\end{lem}

\begin{lem}\label{bilinear-Y} Let $\mathscr{J}(u_1,u_2)=u_1u_2$ or $u_1\bar{u}_2$. For any square $C_1, C_2$ of {sidelength $N$ in $\mathbb{Z}^2$}, we have
	\begin{equation}\label{est:bilinear-Y}
		\|\widetilde{\mathscr{J}}( P_{C_1} u_1, P_{C_2}u_2) \|_{{L_{t,x}^{p}([0,1]\times \mathbb T^{2})}} \lesssim   N^{2-\frac {4}{p}}  \|P_{C_1}u_1 \|_{Y^0} \|P_{C_2}u_2\|_{Y^0},\quad \forall p>2.
	\end{equation}
\end{lem}

\begin{lem}\label{DSbilinear-Y} Let $\mathscr{J}=\mathscr{J}_{\mathrm{DS}}$ or $\sqrt{\mathscr{J}_{\mathrm{DS}}}$. For any square $C_1, C_2$ of sidelength $N$ in $\mathbb{Z}^2$, we have
	\begin{equation}\label{est:DSbilinear-Y}
		\|\mathscr{J}( P_{C_1} u_1, P_{C_2}u_2) \|_{{L_{t,x}^{2}([0,1]\times \mathbb T^{2})}} \lesssim   N^\varepsilon  \|P_{C_1}u_1 \|_{Y^0} \|P_{C_2}u_2\|_{Y^0}.
	\end{equation}
	for any $\varepsilon>0.$
\end{lem}
These three lemmas follow from the free wave estimates in Theorem~\ref{thm:main sharp Strichartz estimate on d-dimensional tori}, Lemma~\ref{lemv2: Bilinear L2 estimate for projection on diagonal} and~\ref{lem:refinedLp}, together with the Galilean transformation, extension and embedding properties in subsection~\ref{subsec:function-space}.   We caution that for \eqref{est:DSbilinear-Y}, we will use interpolation in Proposition~\ref{prop:interpolation} which gains extra $(\ln N)^2$, but that will be absorbed  into the  $N^\varepsilon$  bound.

The following trilinear estimate will play a key role in the proof of well-posedness.
\begin{lem}\label{trilinear-Y}  Consider $u_i$ with frequency $N_i$,  $N_1\geq N_2\geq N_3$. Then we have
	\begin{equation}\label{est:trilinear-Y} \|u_1u_2u_3\|_{L^2_{t,x}([0,1]\times \mathbb{T}^2)}\lesssim N_2^{\frac12}N_3^{\frac12} \prod_{i=1}^3 \|u_i\|_{Y^0}.
	\end{equation}
\end{lem}
\begin{proof}
	Decompose $\mathbb Z^2 = \bigcup_j  C_j$ into almost disjoint squares with side length $N_2$ and write
	\begin{equation}
		u_1u_2u_3 = \sum_{C_j}  ( P_{C_j}  u_1)u_2u_3.
	\end{equation}
	Their Fourier supports are finitely overlapped, hence we have the almost orthogonality
	\begin{equation}\label{trilinear-orth-decom}
		\| u_1u_2u_3 \|_{L^2_{t,x}} ^2\approx \sum_{C_j} \| ( P_{C_j}  u_1)u_2u_3\|_{L^2_{t,x}}^2.
	\end{equation}
	We use $\mathscr{J}(u_1,u_2)=u_1u_2$ in this proof.
	From Lemma~\ref{linear-Y} and \ref{bilinear-Y}, we get for any $p>2$
	\begin{align}
		\| \widetilde{\mathscr{J}}(P_{C_j}  u_1, u_2)u_3\|_{L^2_{t,x}} & \leq \|\widetilde{\mathscr{J}}(P_{C_j}  u_1, u_2) \|_{L^p_{t,x}} \| u_3 \|_{L^{2p/(p-2)}_{t,x}} \\& \lesssim N_2^{2-4/p} N_3^{4/p-1} \| P_{C_j}u_1\|_{Y^0} \| u_2 \|_{Y^0} \| u_3\|_{Y^0}. \label{est:trilinear off-diagonal}
	\end{align}
	From \eqref{estv2:trilinear diagonal orthogonality} we estimate that
	\begin{align}
		 & \phantom{{}\lesssim{}}\sum_{s\in\mathbb{Z}} \| \mathscr{J}(P_{\ell_{\pm}^{s}} P_{C_j} u_1,P_{\ell_{\pm}^{s}}u_2)u_3\|_{L^2_{t,x}}
		\\& \lesssim  \sum_{s\in\mathbb{Z}} \| \mathscr{J}(P_{\ell_{\pm}^{s}} P_{C_j} u_1,P_{\ell_{\pm}^{s}}u_2)u_3\|_{L^\infty_t L^2_x}
		\\& \lesssim N_2^{1/2}N_3^{1/2}
		\sum_{s\in\mathbb{Z}} \|P_{\ell_{\pm}^{s}}P_{C_j} u_1 \|_{L^\infty_t L^2_x} \| P_{\ell_{\pm}^{s}} u_2 \|_{L^\infty_tL^2_x} \| u_3 \|_{L^\infty_tL^2_x} \\& \lesssim N_2^{1/2} N_3^{1/2} \| P_{C_j} u_1\|_{Y^0} \| u_2\|_{Y^0} \| u_3 \|_{Y^0}.\label{est:trilinear diagonal}
	\end{align}
	The conclusion follows by taking $p$ close to $2$, and combining  estimates \eqref{trilinear-orth-decom}\eqref{est:trilinear off-diagonal} and \eqref{est:trilinear diagonal} .

\end{proof}
\begin{remark}
	If we only use the sharp Strichartz estimate \eqref{est:main sharp Strichartz estimate on d-dimensional tori}
	combining with the finer frequency decomposition to  get that (still assume $N_1\geq N_2\geq N_3$)
	\begin{align}
		\|  u_1u_2u_3\|_{L^2_{t,x}([0,1]\times\mathbb{T}^2)}  \leq & {} \|u_1u_2\|_{L^p_{t,x}([0,1]\times\mathbb{T}^2)} \| u_3 \|_{L^{2p/(p-2)}_{t,x}([0,1]\times\mathbb{T}^2)} \\\lesssim{}&  N_2^{\max\{ 2-\frac 4p, 1-\frac 1p \}}  N_3^{ \max\{ \frac 4p-1,\frac1p\}} \| u_1 \|_{Y^0} \| u_2\|_{Y^0} \| u_3\|_{Y^0},
	\end{align}
	with
	\begin{align}
		\max\Big\{ 2-\frac 4p, 1-\frac 1p \Big\} + \max\Big\{ \frac 4p-1,\frac1p \Big\} = 1 + \Big|\frac 3p-1\Big|.
	\end{align}
	Choosing $p=3$ (corresponds to using $L^6$ Strichartz estimate) we get $N_2^{\frac23}N_3^{\frac13}$ bound.
	This avoids additional derivative loss (match with scaling), yet we still face  the summation problem for $s=2/3$ when working with septic HNLS. Thus it is necessary to get a trilinear improvement. In ~\cite{herr2011}, it was done based on orthogonality in time frequency. Here we exploited the structure of hyperbolic Schr\"{o}dinger equation and splitted the
	off-diagonal part out, which allows us to utilize the dashed lines in Figure~\ref{figure: relation between power index and p} without extra derivative loss, leading to a trilinear improvement.
\end{remark}

\subsection{Estimate for the Duhamel term}
\begin{prop}\label{est: nonlinear Duhamel terms}
	Let $0<T<1$. We have  the following estimates on $\mathbb{T}^2$.
	\begin{enumerate}
		\item For $u_1,\dots,u_7\in X^{2/3}([0,T))$,
		      \begin{equation}
			      \label{estv2:multilinear estimate for 7 order HNLS Duhamel term}
			      \left\| \int_0^t \me^{\mi (t-t')\Box} (u_1\bar u_2u_3\bar u_4u_5\bar u_6u_7) \dif t' \right\|_{X^{2/3}([0,T))} \lesssim \prod_{i=1}^{7} \| u_i\|_{X^{2/3}{([0,T))}}.
		      \end{equation}
		\item For any $s>1/2$ and $u_1,u_2,u_3\in {X^s([0,T))}$,
		      \begin{equation}
			      \label{estv2:multilinear estimate for cubic HNLS Duhamel term}
			      \left\| \int_0^t \me^{\mi (t-t')\Box} (u_1\bar u_2u_3) \dif t' \right\|_{{X^s([0,T))}} \lesssim_{s} \prod_{i=1}^{3} \| u_i\|_{{X^s([0,T))}}.
		      \end{equation}
		\item For any $s>0$ and $u_1,u_2,u_3\in {X^s([0,T))}$,
		      \begin{equation}
			      \label{estv2:multilinear estimate for DS Duhamel term}
			      \left\| \int_0^t \me^{\mi (t-t')\Box} (\mathscr{J}_{\mathrm{DS}}(u_1\bar u_2)u_3) \dif t' \right\|_{{X^s([0,T))}} \lesssim_{s} \prod_{i=1}^{3} \| u_i\|_{{X^s([0,T))}}.
		      \end{equation}
	\end{enumerate}
	Here the  implicit constant does not depend on $T$.
\end{prop}
\begin{proof}
	For \eqref{estv2:multilinear estimate for 7 order HNLS Duhamel term}, it suffices to show that for any $u_0 \in Y^{-2/3}([0,T))$, we have
	\begin{equation}
		\left| \int_{[0,T)\times \mathbb T^2} u_0\prod_{i=1}^{7}u_i \dif x\dif t \right| \lesssim \| u_0\|_{Y^{-2/3}([0,T))} \prod_{i=1}^{7} \| u_i \|_{X^{2/3}([0,T))}.
	\end{equation}
	We apply Littlewood-Paley decomposition to each $u_i$ to write
	\begin{equation}
		u_i = \sum_{N_i\text{ dyadic}} P_{N_i}u_i = \sum_{N_i\text{ dyadic}} u_{N_i}^{(i)},
	\end{equation}
	hence it suffices to estimate
	\begin{equation}
		\sum_{N_0,\dots,N_{2k+1}} \left| \int_{[0,T)\times \mathbb T^2} u_{N_0}^{(0)}\prod_{i=1}^{7} u_{N_i}^{(i)} \dif x\dif t \right|.
	\end{equation}
	In order to make the integral non-zero, we must have that the two highest frequencies are comparable. Due to symmetry, it's harmless to assume $N_1\geq N_2\geq \dots\geq N_{7}$.
	We use H\"{o}lder's inequality and put more weights on lower frequencies to help summation.
	\begin{equation}
		\left| \int_{[0,T)\times\mathbb{T}^2} u_{N_0}^{(0)} \prod_{i=1}^7 u_{N_i}^{(i)} \dif x\dif t \right| \leq \| u_{N_0}^{(0)}u_{N_2}^{(2)}u_{N_4}^{(4)}\|_{L^2_{t,x}} \| u_{N_1}^{(1)}u_{N_3}^{(3)}u_{N_5}^{(5)} \|_{L^2_{t,x}} \| u_{N_6}^{(6)}  u_{N_7}^{(7)} \|_{L^\infty_{t,x}} .
	\end{equation}
	Applying Lemma~\ref{trilinear-Y}, we get
	\begin{align}
		 & \phantom{{}\lesssim{}}\sum_{N_0,\dots,N_{2k+1}} \left| \int_{[0,T)\times \mathbb T^2} u_{N_0}^{(0)}\prod_{i=1}^{2k+1} u_{N_i}^{(i)} \dif x\dif t \right|                                                                                         \\&
		\lesssim \sum_{N_0,\dots,N_{2k+1}} \min(N_0,N_2)^{\frac 12}N_3^{\frac12}N_4^{\frac12}N_5^{\frac12} N_6N_7\prod_{i=0}^7 \| u_{N_i}^{(i)} \|_{Y^0}                                                                                                    \\&
		\lesssim \sum_{N_0,\dots,N_{2k+1}} \left( \frac{N_0}{N_1} \right)^{\frac 23}N_2^{-\frac16}N_3^{-\frac16}N_4^{-\frac16}N_5^{-\frac16} N_6^{\frac13}N_7^{\frac13} \| u_{N_0}^{(0)} \|_{Y^{-\frac 23}}\prod_{i=1}^7 \| u_{N_i}^{(i)} \|_{Y^{\frac 23}} \\&
		\lesssim\| u_{0} \|_{Y^{-\frac 23}}\prod_{i=1}^7 \| u_{i} \|_{Y^{\frac 23}}.
	\end{align}
	For \eqref{estv2:multilinear estimate for cubic HNLS Duhamel term} we apply Lemma~\ref{trilinear-Y} and get
	\begin{align}
		 & \left| \int_{[0,T)\times \mathbb{T}^2} \bar u_0 u_1\bar u_2u_3 \dif x\dif t\right|\leq \sum_{N_0,N_1,N_2,N_3} \| u_{N_0}^{(0)} \|_{L^2_{t,x}} \| u_{N_1}^{(1)}u_{N_2}^{(2)}u_{N_3}^{(3)} \|_{L^2_{t,x}} \\ \lesssim{} &\sum_{N_0,N_1,N_2,N_3} \left( \frac{N_0}{N_1} \right)^{s} N_2^{\frac12-s}N_3^{\frac12-s} \| u_{N_0}^{(0)} \|_{Y^{-s}} \| u_{N_1}^{(1)} \|_{Y^s} \| u_{N_2}^{(2)} \|_{Y^s} \| u_{N_3}^{(3)} \|_{Y^s}
		\\\lesssim{} &\| u_0 \|_{Y^{-s}} \| u_1 \|_{Y^s} \| u_2 \|_{Y^s} \|u_{3} \|_{Y^s}.
	\end{align}
	For \eqref{estv2:multilinear estimate for DS Duhamel term}, we use similar method to estimate
	\begin{equation}
		\label{equationv2:DS 4 linear integral}
		\sum_{N_0,N_1,N_2,N_3} \left| \int_{[0,T)\times\mathbb{T}^2}  u_{N_0}^{(0)} \mathscr{J}_{\mathrm{DS}}(u_{N_1}^{(1)}u_{N_2}^{(2)})u_{N_3}^{(3)} \dif x\dif t\right|,
	\end{equation}
	where the two highest frequencies are comparable.
	We relabel the frequencies $ \{N_0,N_1, N_2,N_3\}$ as $N_{\mathrm{max}} \approx N_{\mathrm{max}}' \gtrsim N_{\mathrm{med}}\gtrsim N_{\mathrm{min}}$. Directly using \eqref{estv2:epsilon improved bilinear estimate for projection on off-diagnoal} implies the bound
	\begin{equation}
		\sum_{N_0,N_1,N_2,N_3} N_{\text{max}}^\varepsilon \| u_{N_0}^{(0)} \|_{Y^0} \| u_{N_1}^{(1)} \|_{Y^0} \| u_{N_2}^{(2)} \|_{Y^0} \| u_{N_3}^{(3)} \|_{Y^0}.
	\end{equation}
	This can be improved by using finer frequency decomposition. We decompose $\mathbb Z^2 = \bigcup_j  C_j$ into almost disjoint squares with side length $N_{\text{med}}$ and write $u_{N_i}^{(i)} = \sum_{C_j} P_{C_j} u_{N_i}^{(i)}$. The decomposition is nontrivial only for the two highest frequencies. Each integral in \eqref{equationv2:DS 4 linear integral} is non-zero only for $\operatorname{dist}(C_j,C_{j'}) \lesssim N_{\mathrm{med}}$, and is bounded by $$ N_{\text{med}}^\varepsilon \| u_{N_0}^{(0)} \|_{Y^0} \| u_{N_1}^{(1)} \|_{Y^0} \| u_{N_2}^{(2)} \|_{Y^0} \| u_{N_3}^{(3)} \|_{Y^0}.$$
	where we used the bilinear estimate Lemma~\ref{DSbilinear-Y}.
	Thus
	\begin{align*}
		 & \left| \int_{[0,T)\times\mathbb{T}^2} u_0\mathscr{J}_{\mathrm{DS}}(u_1\bar u_2)u_3\dif x\dif t \right| \\ \lesssim{} &\sum_{N_0,N_1,N_2,N_3}
		N_{\mathrm{med}}^{\varepsilon-s} N_{\mathrm{min}}^{-s}  \| u_{N_0}^{(0)} \|_{Y^{-s}} \| u_{N_1}^{(1)} \|_{Y^s} \| u_{N_2}^{(2)} \|_{Y^s} \| u_{N_3}^{(3)} \|_{Y^s}.
	\end{align*}
	Take $0<\varepsilon<s$ and apply Cauchy-Schwarz inequality to summations about $N_{\mathrm{min}}$, $N_{\mathrm{med}}$ and $N_{\mathrm{max}} \approx N_{\mathrm{max}}'$ respectively (notice that they are all dyadic integers), we get the  desired bound.
\end{proof}

\subsection{Contractions}
It suffices to prove the map
\begin{equation}
	u \mapsto \mathcal{I}(u)(t) = \int_0^t \me^{\mi (t-t')\Box} F(u) \dif t'
\end{equation}
is a contraction under the metric $d(u,v) = \| u-v \|_{{X^s([0,T))}}$ for sufficiently small $T$, where
\[ F(u) = |u|^6u,\ s=2/3 \text{ or }F(u) = |u|^2u,\ s>1/2\text{ or }F(u) = \mathscr{J}_{\mathrm{DS}}(|u|^2)u,\ s>0.\]
For convenience, we set $k=k(s)=3,1,1$ respectively. Given initial data $\phi\in H^s(\mathbb{T}^2)$ with $\| \phi \|_{H^s(\mathbb{T}^2)} \leq A$, we choose $\delta$ sufficiently small depending on $A$, and $N$ sufficiently large such that $\| P_{>N} \phi \|_{H^s(\mathbb{T}^2))} \leq \delta$.
For any $u,v$ in the set
\begin{multline*}
	D_\phi : = \{ u \in C([0,T);H^s(\mathbb T^2)) \cap X^s([0,T)) \mid \\
	u(0)=\phi,\ \|u\|_{X^s([0,T))} \leq 2A,\ \|  P_{>N} u \|_{X^s([0,T))} \leq 2\delta
	\},
\end{multline*}
we can decompose
\begin{equation}
	F(u)-F(v) = F_1(u,v) + F_2(u,v),
\end{equation}
where $F_1(u,v)$ is a combination of $u-v,P_{\leq N}u,P_{\leq N}v$, and all terms involving $P_{>N}u,P_{>N}v$ appear in $F_2(u,v)$.
Employing Sobolev embeddings and  \cite[Theorem A.12]{Kenig1993}, we estimate that
\begin{align}
	       & \left\| \int_0^t \me^{\mi (t-t')\Box} F_1(u,v) \dif t' \right\|_{X^s([0,T))} \leq CT \| F_1(u,v)\|_{L^\infty H^s}                                                                                                                                                \\
	\leq{} & CT \left( \| u-v \|_{L^\infty H^s} \left(\| P_{\leq N} u\|_{L^\infty_{t,x}}^{2k} + \| P_{\leq N} v \|_{L^\infty_{t,x}}^{2k}\right) \right.                                                                                                                       \\
	       & \qquad + N^s \| u-v\|_{L_t^{\infty\vphantom{/}} L_{x\vphantom{t}}^{2/(1-s)}} \left.\left (\| P_{\leq N} u\|_{L_t^{\infty\vphantom{/}} L_{x\vphantom{t}}^{4k/s}}^{2k}+ \| P_{\leq N} v \|_{L_t^{\infty\vphantom{/}} L_{x\vphantom{t}}^{4k/s}}^{2k} \right)\right) \\
	\leq{} & CTN^{2k(1-s)} (2A)^{2k}\| u-v \|_{{X^s([0,T))}} .\label{est:contraction-low}
\end{align}
While by Proposition~\ref{est: nonlinear Duhamel terms} we have that
\begin{align}
	       & \left\| \int_0^t \me^{\mi (t-t')\Box} F_2(u,v) \dif t' \right\|_{X^s([0,T))}                   \\ \leq {}&
	C \| u-v\|_{X^s} ( \| P_{>N} u \|_{X^s} + \| P_{>N}v \|_{X^s} )( \| u \|_{X^s} + \| v \|_{X^s} )^{2k-1} \\
	\leq{} & C(2A)^{2k-1}(2\delta) \| u-v\|_{X^s([0,T))}. \label{est:contraction-high}
\end{align}
Hence we get that
\begin{equation}
	\left\| \mathcal{I}(u) -  \mathcal{I}(v) \right\|_{{X^s([0,T))}}
	\leq  \frac{1}{10} \| u-v\|_{{X^s([0,T))}}, \label{est: Picard iteration is contraction}
\end{equation}
provided $\delta$ is chosen sufficiently small depending on $A, k$, and $T$ is chosen sufficiently small depending on $A,N$ and $k$.

Next we verify that $\me^{\mi t\Box}\phi + \mathcal{I}(u) \in D_\phi$ for $u\in D_\phi$. First we notice that
\begin{align}
	\| \mathcal{I}(P_{\leq N} u) \|_{{X^s([0,T))}} & \leq C \| F(P_{\leq N}u ) \|_{L_t^1H^s_x} \leq CT \| P_{\leq N} u \|_{L^\infty_t H^s_x} \| P_{\leq N} u \|_{L^\infty_{t,x}}^{2k} \\& \leq CTN^{2k(1-s)}(2A)^{2k+1} \label{est:Duhamel at low frequency}.
\end{align}
Besides, applying \eqref{est: Picard iteration is contraction} for $v=P_{\leq N}u$ we get
\begin{equation}
	\| \mathcal{I}(u) - \mathcal{I}(P_{\leq N}u ) \|_{{X^s([0,T))}} \leq \frac1{10} \| P_{>N}u\|_{{X^s([0,T))}} \leq \frac\delta5. \label{est: difference of Duhamel and low frequency}
\end{equation}
It's not hard to get the desired estimates for $\mathcal{I}(u)$ and $P_{>N}\mathcal{I}(u)$ from \eqref{est:Duhamel at low frequency} \eqref{est: difference of Duhamel and low frequency}.
\subsection{Ill-posedness} The ill-posedness results are proved by the same idea as in \cite{LZ2025}, based on the fact that the hyperbolic paraboloid contains a vector subspace. Let $V\subset \mathbb{R}^d$ be a vector subspace such that $h(\xi)=0$ for all $\xi\in V$. We denote $V_{N} = V \cap \mathbb{Z}^d \cap [1,N]^d$ and consider the function
\begin{equation}
	\phi_N(x) =  \sum_{\xi\in V_{N}} \frac 1{|\xi|^{2s}}\me^{2\pi\mi x\cdot \xi}, \quad x\in \mathbb{T}^d,\quad s=\frac {\dim V}2.
\end{equation}
It's easy to calculate that $\| \phi_N\|_{H^s(\mathbb{T}^d)} \approx (\log N)^{1/2}$. On the other hand, for $\xi \in V_{ (2k+1)N}$
\begin{align}
	\widehat{|\phi_N|^{2k}\phi_N} (\xi) & = \sum_{\xi_1-\xi_2+\dots+\xi_{2k+1}=\xi} \prod_{i=1}^{2k+1} \frac1{|\xi_i|^{2s}} \gtrsim \frac1{|\xi|^{2s}} \prod_{i=1}^{2k} \sum_{\xi_i \in V_{\frac{|\xi|}{2k+1}}} \frac 1{|\xi_i|^{2s}} \\& \gtrsim \frac{ (\log |\xi|)^{2k}}{|\xi|^{2s}}.
\end{align}

As a consequence
\begin{align}
	\| |\phi_N|^{2k} \phi_N \|_{H^s(\mathbb{T}^d)} \gtrsim \left( \sum_{\xi } \frac{(\log|\xi|)^{4k}}{|\xi|^{2s}}  \right)^{1/2} \gtrsim (\log N)^k \| \phi_N \|_{H^s(\mathbb{T}^d)}^{2k+1}.
\end{align}
It's easy to verify that
\begin{equation}
	\int_0^t \me^{\mi (t-t')\Box} \Big( |\me^{\mi t'\Box}\phi_N|^{2k}\me^{\mi t'\Box} \phi_N \Big) \dif t' = t |\phi_N|^{2k}\phi_N,
\end{equation}
\begin{equation}
	\int_0^t \me^{\mi (t-t')\Box} \Big( \mathscr{J}_{\mathrm{DS}}(|\me^{\mi t'\Box}\phi_N|^{2})\me^{\mi t'\Box} \phi_N \Big) \dif t' =0.
\end{equation}
Hence we see that  the first Picard iteration is unbounded on $H^{1/2}(\mathbb T^2)$ for quintic HNLS and DS with $\sigma_2\neq \gamma/(1+\alpha)$, as well as unboundness on $H^1(\mathbb{T}^4)$ for cubic symmetric HNLS.  This concludes the  proof of ill-posedness results in Theorem~\ref{thm:cauchyProblemHNLS} and ~\ref{thm:LWP for DS}.

\end{document}